\documentclass{article}

\usepackage{arxiv}
\usepackage{graphicx,units,amsmath}

\usepackage[dvipsnames]{xcolor}
\usepackage{graphicx,float,amsmath,amsfonts,enumitem}
\usepackage{booktabs,rotating,multirow,array}
\usepackage[percent]{overpic}
\usepackage{adjustbox}
\usepackage[utf8]{inputenc}
\usepackage{hyperref}
\usepackage{scalerel}
\usepackage{mathtools}
\usepackage[square,sort,comma,numbers]{natbib}
\usepackage{amsthm}
\usepackage{caption}
\usepackage{multirow}
\usepackage{algorithm}
\usepackage{algpseudocode}

\usepackage{tabularx}

\newtheorem{theorem}{Theorem}
\newtheorem{lemma}{Lemma}

\newtheorem{definition}{Definition}
\newtheorem{corollary}{Corollary}

\title{Fast stable finite difference schemes for nonlinear cross-diffusion}

\author{Diogo Lobo \\ diogo.lobo@mat.uc.pt \\ CMUC, Department of Mathematics,
		University of Coimbra}

\date{May 2021}

\renewcommand{\headeright}{}
\renewcommand{\undertitle}{}
\renewcommand{\shorttitle}{Fast stable finite difference methods for nonlinear cross-diffusion}

\begin{document}
	
	\maketitle
	
	\begin{abstract}
		The dynamics of cross-diffusion models leads to a high computational complexity for implicit difference schemes, turning them unsuitable for tasks that require results in real-time. We propose the use of two operator splitting schemes for nonlinear cross-diffusion processes in order to lower the computational load, and establish their stability properties using discrete $L^2$ energy methods. Furthermore, by attaining a stable factorization of the system matrix as a forward-backward pass, corresponding to the Thomas algorithm for self-diffusion processes, we show that the use of implicit cross-diffusion can be competitive in terms of execution time, widening the range of viable cross-diffusion coefficients for \textit{on-the-fly} applications.
	\end{abstract}
	
	\section{Introduction}
	
	Cross-diffusion models consist of evolutionary systems of diffusion type for at least two real-valued functions, where the evolution of each function is not independent of the others. Their use is widespread in areas like population dynamics (see \cite{Andreianov11} and references therein), but has recently attracted some interest in the image processing community \cite{AraEtAl17a,AraEtAl17b,AraEtAl17c,Barbeiro20} as a natural extension to the complex diffusion methods proposed by Gilboa et al. \cite{GilSocZee04}, where the image is represented by a complex function and the filtering process is governed by a nonlinear PDE of diffusion type with a complex-valued diffusion coefficient. This equation can be written as a cross-diffusion system for the 1) real and 2) imaginary parts of the image, and enhanced imaging possibilities emerge if we drop the complex point of view and work with 1) the image to be processed and 2) information about that image (e.g., edge locations)  \cite{AraEtAl17c,Barbeiro20}.
	
	It is known that implicit schemes have better stability properties when compared to their explicit counterparts \cite{Suli14}, allowing larger time steps and, in the particular case of cross-diffusion, a wider range of diffusion coefficients \cite{AraEtAl17c,Barbeiro20}. However, the computational complexity of traditional implicit finite difference schemes turns their use impracticable for tasks that aim for \textit{on-the-fly} results, such as image processing in medical contexts. In this field, normally used strategies to speed-up computations include operator splitting techniques, multigrid schemes or, more recently, fast explicit diffusion \cite{Weickert98,Weickert16}. Several splitting techniques have been proposed since the middle of the previous century. First introduced by Douglas, Peaceman and Rachford \cite{Douglas56,Douglas64}, they were further explored by others such as Marchuk or Yanenko \cite{Yanenko71,Marchuk90}. Overall, their aim is the reduction of a multidimensional problem to a sequence of one-dimensional problems, in order to obtain simplified linear systems that can be efficiently processed by computers. In the case of the linear diffusion equation in two dimensions, the factorization of a $5$ point stencil spatial operator leads to a series of tridiagonal systems that can be processed by Thomas algorithm \cite{Quarteroni07,Higham02}. 
	
	In this report we propose the generalization of two splitting techniques commonly used in the imaging community, usually referred to as Additive Operator Splitting (AOS) and Additive Multiplicative Operator Splitting (AMOS), to the nonlinear cross-diffusion case. AOS schemes are very popular since their introduction to the PDE-related imaging field by Weickert \cite{Weickert98}, but their source dates back a few years to the work of Tai and Neitaanmaki \cite{Tai91}. These authors noted that the classical splitting-up methods can not be used for parallel processors as the computation of the current fractional step always requires the knowledge of the previous fractional step. They proposed a new splitting where the computation of the fractional steps are independent of each other, and therefore parallelizable. Further work, now also with Lu \cite{Lu91}, explored convergence estimates for linear and nonlinear elliptic problems and for linear and quasilinear evolution equations. AMOS schemes belong to the family of multiplicative locally one-dimensional schemes, the ones where each fractional step computation requires the value of the previous fractional step. They rely on the same factorization of AOS, but each operator is now applied on top of each other (being called multiplicative). As, in general, the split operators do not commute, the final result depends on the order of the one-dimensional operators. This is a major drawback in some applications, particularly in image processing, as one usually aims for rotation invariant filters. In order to overcome this problem, Barash and Kimmel \cite{Bara01} added a symmetric setting by applying a multiplicative scheme in all possible orders and averaging the results. The generalization of these techniques to our problem still has some drawbacks, as the numerical system to solve remains dense due to the interdependence of the evolution functions inherent to the cross-diffusion model. To overcome this issue we reorder the unknowns to obtain a tridiagonal by blocks system matrix. This new equivalent system allows for a block LU factorization solver to act as a forward-backward pass, similar to the process given by Thomas algorithm \cite{Higham02}.
	
	Regarding previous works on numerical schemes for cross-diffusion processes, the available literature is scarce and scattered. For general nonlinear cross-diffusion systems, a linear scheme based on the nonlinear Chernoff formula was proposed by Murakawa \cite{Murakawa11} and a convergent finite volume method was introduced by Andreainov et al. \cite{Andreianov11}. This technique was further explored again by Murakawa in \cite{Murakawa2017}, where explicit algebraic corrections at each time step were employed in order to achieve unconditional stability. However, to the best of our knowledge, the only attempt for computational competitive schemes comes from Beauregard and Padgett \cite{Beauregard18}, where a Douglas-Gunn splitting finite difference scheme was deployed to achieve conditional stability, and a  computation time scaling as $\mathcal{O}(N^{1.7568})$ was obtained experimentally for one dimensional processes.
	
	The manuscript is organized in the following way. In Section \ref{section:main_method} we describe the main model. We present the standard finite difference $\theta$-method for the two-dimensional case, and we derive its stability properties. In Section \ref{section:splittings} we introduce the AOS and AMOS schemes for the cross-diffusion model, and we extend the stability results for these schemes and discuss their computational implementations. Still in this section, we describe how to design the system matrix as tridiagonal by blocks, and discuss the stability conditions of the correspondent factorization. In Section \ref{section:3D} we extend the results to the three-dimensional case. Finally, in Section \ref{section:speedups} we present the theoretical score of speedups and some numerical experiments.
	
	\section{Model and standard finite difference scheme}
	\label{section:main_method}
	\subsection{Nonlinear cross-diffusion system}
	We consider a nonlinear cross-diffusion with semi-linear reaction process that can be represented by a two-component vector field,  ${\bf w}=(u,v)^\top$, satisfying the  system\begin{equation} \label{cross-diffusion}
		\begin{cases}
			u_t=\nabla \cdot(d^1(u,v,t)\nabla u + d^2(u,v,t)\nabla v) - \lambda_1(u,v,t) (u-u^0)\text{ in }\Omega\times\mathbb{R}^+,\\
			v_t=\nabla \cdot(d^3(u,v,t)\nabla u + d^4(u,v,t)\nabla v) - \lambda_2(u,v,t) (v-v^0)\text{ in }\Omega\times\mathbb{R}^+,\\
			u({\bf x},0)=u^0({\bf x}),\text{ }v({\bf x},0)=v^0({\bf x})\text{ in } \Omega,\\
			u_\eta =0,\text{ }v_\eta=0\text{ on } \Gamma\times\mathbb{R}^+,
		\end{cases}
	\end{equation}
	where  $\Omega=(a_{1},b_{1})\times (a_{2},b_{2})\subset\mathbb{R}^{2}$ is the domain of interest, $u^0$ and $v^0$ are the given initial conditions for $u$ and $v$ and $\eta$ denotes the outward normal vector to the boundary  $\Gamma=\partial \Omega$.   
	The cross-diffusion matrix  of the model is given by
	\begin{equation}\label{cross-diffusion matrix}
		D(u,v,t)={\begin{bmatrix}d^1({u,v},t)&d^2({u,v},t)\\d^3({u,v},t)&d^4({u,v},t)\end{bmatrix}},
	\end{equation}
	where $d^\ell,$ $\ell=1,\dots,4$, are the so called influence functions, and $\lambda_1,\lambda_2$ are non-negative real valued bounded functions. Before moving on to the theoretical results for finite difference implementations of (\ref{cross-diffusion}), we start by grounding ourselves with some definitions. More precisely, we state our understanding of positive definite matrices \cite{Horn12}:
	
	\begin{definition}
		Let $A$ be a complex valued square matrix of size $n$. We say that $A$ is positive (semi-)definite if
		$$x^\top A x > 0\quad (\geq 0) ,\quad \forall x\in \mathbb{C}^n.$$
	\end{definition}

	The standard definition for positive (semi-)definiteness of a matrix $A$ implies that $A$ is automatically hermitian \cite{Dym08}. We will make an intensive use of a similar class of matrices, where the implicit necessity for hermiticity is no longer required:
	
	\begin{definition}
		Let $A$ be a real valued square matrix of size $n$. We say that $A$ is \emph{positive (semi-)definite, not necessarily symmetric}, if
		$$x^\top A x > 0\quad (\geq 0) ,\quad \forall x\in \mathbb{R}^n.$$
	\end{definition} 

	We also unambiguously define the notion of stable finite difference schemes following Jovanovi{\'c} and S{\"u}li \cite{Suli14}.
	
	\begin{definition}
		Let $(\bar{U}^m,\bar{V^m})$ be the $m$-step of a finite difference scheme with space-mesh parameter $h$ and time step $\Delta t$ applied to problem (\ref{cross-diffusion}). Let $\|\cdot\|_{\phi(h)}$ be a mesh-dependent norm involving mesh points of $\Omega$. We say that the scheme is unconditionally stable if there is $C>0$ independent of the mesh size such that 
		\begin{equation}
			\label{suli:stab_cdt}
			\|(\bar{U}^m,\bar{V^m})\|_{\phi_h}\leq C\|(U^0,V^0)\|_{\phi_h}
		\end{equation}
		holds for any choice of parameters $h$ and $\Delta t$, where $(U_0,V_0)$ are the initializations $u^0$ and $v^0$ taken at the mesh points. If (\ref{suli:stab_cdt}) holds only for $\Delta t\leq C(h)$, where $C(h)$ is a function of $h$, we say that the scheme is conditionally stable.
	\end{definition}

	\subsection{Explicit and semi-implicit implementations and stability results}
	
	Let the domain $\overline{\Omega}=\Omega\cup\Gamma$ be discretized by the points ${\bf x}_{{\bf j}}=(x_{j_{1}},x_{j_{2}})$, where
	\[
	x_{j_{1}}=a_{1}+h_{1}j_{1},\; x_{j_{2}}=a_{2}+h_{2}j_{2},\quad j_{k}=0,1,\ldots,N_{k},\]
	\[ h_{k}=\frac{b_{k}-a_{k}}{N_{k}}, \; k=1,2,
	\]
	for two given integers $N_{1}, N_{2}\geq 1$, ${\bf j}=(j_{1},j_{2})$ and ${\bf h}=(h_{1},h_{2})$ . This spatial mesh on $\overline{\Omega}$ is denoted by $\overline{\Omega}_{\bf h}$ and  $\Gamma_{\bf h}=\Gamma \cap \overline{\Omega}_{\bf h}$.  
	Points halfway between two adjacent grid points are denoted by 
	${\bf x}_{{\bf j}\pm (1/2){\bf e}_k}={\bf x}_{\bf j} \pm  \frac{h_k}{2}{\bf e}_{k}$, $k=1,2$, where $\{{\bf e}_{1},{\bf e}_{2}\}$ 
	is the $\mathbb{R}^2$ canonical basis, that is, ${\bf e}_{k}$ is the standard basis unit vector in the $k$th direction. 
	
	For the discretization in time, we consider a mesh with time step $\Delta t$,
	$0=t^0<t^1<t^2<\ldots$, where $t^{m+1}-t^m=\Delta t$. 
	
	We denote by $Z_{\bf j}^m$ the value of a mesh function $Z$ at the point $({\bf x}_{{\bf j}}, t^m)$. 
	For the formulation of the finite difference approximations, we use the centered finite difference quotients in the $k$th spatial direction, for $k=1,2$,
	\[
	\delta_{k}Z_{\bf j} = \frac{Z_{{\bf j}+(1/2){\bf e}_{k}}-Z_{{\bf j}-(1/2){\bf e}_{k}}}{h_{k}}, \]
	\[ \delta_{k}Z_{{\bf j}+(1/2){\bf e}_{k}} = \frac{Z_{{\bf j}+{\bf e}_{k}}-Z_{\bf j}}{h_{k}}.\label{sm2a}
	\]

	An initial distribution $(U^0,V^0)$ is required, given by two real-valued functions $U^0,V^0: \overline \Omega_{\bf h}\rightarrow\mathbb{R}$. Let ${\bf W}^{m}_{\bf j}=(U_{{\bf j}}^{m},V_{{\bf j}}^{m})^\top$, such that ${\bf x}_{\bf j} \in  \overline \Omega_{\bf h}$. Given the initial solution ${\bf W}^{0}_{\bf j}=(U^0_{\bf j},V^0_{\bf j})$,  the numerical solution of (\ref{cross-diffusion}) at the time $t^{m+1}$ is obtained considering the following finite difference scheme:
	
	\begin{equation}
		\label{finite_difference_scheme}
		\begin{cases}
			\frac{U_{\bf j}^{m+1}-U_{\bf j}^{m}}{\Delta t} =
			\sum\limits_{k=1}^{2}{ \delta_{k}\left(d^1({\bf W}^m)_{\bf j}^{m+\theta}\delta_{k}U_{{\bf j}}^{m+\theta}+d^2({\bf W}^m)_{\bf j}^{m+\theta}\delta_{k}V_{{\bf j}}^{m+\theta}\right)
				-\lambda_1({\bf W}^m)_{\bf j}^{m+\theta} (U_{\bf j}^{m+\theta}-U_{\bf j}^0),} \\
			\frac{V_{\bf j}^{m+1}-V_{\bf j}^{m}}{\Delta t} = 
			\sum\limits_{k=1}^{2}{ \delta_{k}\left(d^3({\bf W}^m)_{\bf j}^{m+\theta}\delta_{k}U_{{\bf j}}^{m+\theta}+d^4({\bf W}^m)_{\bf j}^{m+\theta}\delta_{k}V_{{\bf j}}^{m+\theta}\right)-\lambda_2({\bf W}^m)_{\bf j}^{m+\theta} (V_{\bf j}^{m+\theta}-V_{\bf j}^0),}
		\end{cases}
	\end{equation}
	where
	\begin{equation}\label{definition_D}
		\begin{aligned}
			&d_\ell({\bf W}^m)_{{\bf j}\pm (1/2){\bf e}_{k}}^{m+\theta}=\frac{d_\ell({\bf W}^m_{{\bf j}},t^{m+\theta})+ d_\ell({\bf W}^m_{{\bf j}\pm{\bf e}_{k}},t^{m+\theta})}{2},\quad \ell=1,\dots,4,\\
			&\lambda_i({\bf W}^m)_{\bf j}^{m+\theta}=\lambda_i({\bf W}^m_{\bf j},t^{m+\theta}), \quad i=1,2,\\
			&Z_{\bf j}^{m+\theta}=\theta Z_{\bf j}^{m+1}+(1-\theta)Z_{\bf j}^m,\quad Z=U,V,
		\end{aligned}
	\end{equation}
	and $\theta \in [0,1]$ corresponding to explicit and semi-implicit implementations, respectively when $\theta=0$ and $\theta=1$, and to a semi-implicit Crank-Nicholson type implementation when $\theta=\frac{1}{2}$. 
	
	For each $x_{\bf j}=(x_{j_1},x_{j_2})\in \bar{\Omega}_{\bf h}$, we define the rectangle $\square_{\bf j}=(x_{j_1},x_{j_1+1})\times(x_{j_2},x_{j_2+1})$ and denote by $|\square_{\bf j}|$ the measure of $\square_{\bf j}$.   We consider the discrete  $L^2$ inner products 
	\begin{equation}\label{inner_product1_2D}
		\begin{aligned}
			({U},{ V})_h&=\sum_{\square_{\bf j}\subset \Omega} \frac{|\square_{\bf j}|}{4} \left({ U}_{j_1,j_2}{ V}_{j_1,j_2} +{ U}_{j_1+1,j_2}{{ V}}_{j_1+1,j_2}\right.\\
			&\quad \left.+{U}_{j_1,j_2+1}{V}_{j_1,j_2+1}+{ U}_{j_1+1,j_2+1}{{ V}}_{j_1+1,j_2+1}\right),
		\end{aligned}
	\end{equation}
	\begin{equation*}
		({U},{V})_{h_1^*}=\sum_{\square_{\bf j}\subset \Omega} \frac{|\square_{\bf j}|}{2} \left({U}_{j_1+1/2,j_2}{{V}}_{j_1+1/2,j_2}
		+{ U}_{j_1+1/2,j_2+1}{ V}_{j_1+1/2,j_2+1}\right),
	\end{equation*}
	\begin{equation*}
		({ U},{ V})_{h_2^*}=\sum_{\square_{\bf j}\subset \Omega} \frac{|\square_{\bf j}|}{2} \left({ U}_{j_1,j_2+1/2}{{ V}}_{j_1,j_2+1/2} +{ U}_{j_1+1,j_2+1/2}{{ V}}_{j_1+1,j_2+1/2}\right).
	\end{equation*}
	Their correspondent norms are denoted by $\|.\|_h$, $\|.\|_{h_1^*}$ and  $\|.\|_{h_2^*}$, respectively.  For ${\bf W}=({U},{ V})^\top$ we define $\|{\bf W}\|_h^2=\|{ U}\|_h^2+\|{ V}\|_h^2$.

	To simplify the notation and when it is clear from the context, we will write $d_\ell$ instead of $d_\ell({\bf W}^m)^{m+\theta},$ or $d_\ell({\bf W}^m)_{\bf j}^{m+\theta}$, for $\ell=1,2,3,4$.

	Theorem \ref{thrm:L2_stab_conditions} states the stability conditions for this finite difference scheme, which is an extension of a previous result presented in \cite{Barbeiro20}. Define $\lambda_\text{max}$ as the supreme of both $\lambda_1$ and $\lambda_2$.
	
	\begin{theorem}\label{thrm:L2_stab_conditions}
		If the cross-diffusion matrix (\ref{cross-diffusion matrix}) is such that, for any $U,V\in\mathbb{R}^{N_1\times N_2}$,
		\begin{equation}\label{cdts:functions_semiimplicit}
			\sum\limits_{k=1}^2 (d^1\delta_k U, \delta_k U)_{h_k^*}+(d^2\delta_k V, \delta_k U)_{h_k^*}+(d^3\delta_k U, \delta_k V)_{h_k^*}+(d^4\delta_k V, \delta_k V)_{h_k^*}\geq 0
		\end{equation}
		then the scheme (\ref{finite_difference_scheme}) is unconditionally stable for $\theta\in[\frac{1}{2},1]$.
		
		If the cross-diffusion matrix (\ref{cross-diffusion matrix}) is such that, for any $U,V\in\mathbb{R}^{N_1\times N_2}$,
		\begin{equation}
			\begin{aligned}
				\label{cdts:functions_explicit}
				\sum\limits_{k=1}^2 & (d^1\delta_k U, \delta_k U)_{h_k^*}+(d^2\delta_k V, \delta_k U)_{h_k^*}+(d^3\delta_k U, \delta_k V)_{h_k^*}+(d^4\delta_k V, \delta_k V)_{h_k^*}\\
				& -\big( \frac{4\Delta t}{h_k^2}\big((1+\eta_{1k})(\|d^1\delta_{k}U\|_{h_k^*}^2+\|d^2\delta_{k}V\|_{h_k^*}^2 +2(d^1\delta_{k}U,d^2\delta_{k}V)_{h_k^*})\\
				& +(1+\eta_{2k})(\|d^3\delta_{k}U\|_{h_k^*}^2+\|d^4\delta_{k}V\|_{h_k^*}^2+2(d^3\delta_{k}U,d^4\delta_{k}V)_{h_k^*}) \big) \geq 0,
			\end{aligned}
		\end{equation}
		for some $\eta_{1k},\eta_{2k} >0$, $k=1,2$, then the scheme (\ref{finite_difference_scheme}) is unconditionally stable for the case $\theta=0$.
	\end{theorem}
	\begin{proof}
		Let $m\in\mathbb{N}$ be any iteration of the finite difference scheme (\ref{finite_difference_scheme}). Let us multiply both members by $U^{m+\theta}$ and $V^{m+\theta}$, respectively, according to the discrete inner product $(\cdot,\cdot)_h$, use summation by parts, the non-negativity of the reaction multiplier, and the discrete Duhamel principle \cite{ChanShen87} to obtain, for the case $\theta\in [\frac{1}{2},1]$,
		\begin{equation*}
			\|{\bf W}^{m+1}\|_h^2 \leq  e^{2(\theta ^2 + (1-\theta)^2)\tilde{\epsilon}^{-1}\epsilon t^{m+1}} \Big(1+t^{m+1}\lambda_\text{max}\epsilon^{-1}\tilde{\epsilon}^{-1}\Big) \|{\bf W}^0\|_h^2,
		\end{equation*}
		with $\tilde{\epsilon}<1-2\Delta t \epsilon\theta^2$, for some $\epsilon>0$. For the case $\theta=0$, take the same steps and the norm $\|\cdot\|_h$ on both sides of (\ref{finite_difference_scheme}) to obtain
		\begin{equation}
			\label{duhamel}
			\begin{aligned}
				&\|{\bf W}^{m+1}\|_h^2\leq  e^{a_\epsilon t^{m+1}}\Big(1+t^{m+1}b_\epsilon\Big)\|{\bf W}^0\|_h^2,
			\end{aligned}
		\end{equation}
		with $a_\epsilon = 2\Delta t(\epsilon + \Delta t \zeta)$ and $b_\epsilon = 2(\frac{\lambda_{\text{max}}^2}{2\epsilon}+\Delta t\zeta)$, where $\zeta=\max \{1+\eta_1^{-1},1+\eta_2^{-1}\}\lambda_\text{max}^2$.
	\end{proof}
	
	We supplement this result with a series of corollaries related to common instances of cross-diffusion matrices \cite{AraEtAl17a,AraEtAl17b,AraEtAl17c,BerEtAl10,GilSocZee04}.
	
	\begin{corollary}
		If the cross-diffusion matrix (\ref{cross-diffusion matrix}) is semi-positive definite, not necessarily symmetric, then the method is unconditionally stable for any $\theta \in [\frac{1}{2},1]$.
	\end{corollary}
	\begin{proof}
		Notice that we can write (\ref{cdts:functions_semiimplicit}) as
		\begin{equation*}
			\sum_{k=1}^2 \sum_{\square_{\bf j}\subset \Omega} \sum_{i=0,1}  \frac{|\square_{\bf j}|}{2} \begin{bmatrix}
					\delta_k U_{{\bf j}+e_k/2+ie_l} \\
					\delta_k V_{{\bf j}+e_k/2+ie_l}
				\end{bmatrix}^\top
				\begin{bmatrix}
					d^1({\bf W})_{{\bf j}+e_k/2+ie_l} & d^2({\bf W})_{{\bf j}+e_k/2+ie_l}\\
					d^3({\bf W})_{{\bf j}+e_k/2+ie_l} & d^4({\bf W})_{{\bf j}+e_k/2+ie_l}
				\end{bmatrix}\begin{bmatrix}
					\delta_k U_{{\bf j}+e_k/2+ie_l} \\
					\delta_k V_{{\bf j}+e_k/2+ie_l}
			\end{bmatrix},
		\end{equation*}
		where $e_l$ is the direction orthogonal to $e_k$. If the cross-diffusion matrix is semi-positive definite, taking (\ref{definition_D}) into account gives that the summand is non negative.
	\end{proof}
	
	\begin{corollary}
		If the cross-diffusion matrix (\ref{cross-diffusion matrix}) is of the form
		$$\begin{bmatrix}
				d^1({\bf W}) & d^2({\bf W})\\
				d^3({\bf W}) & d^4({\bf W})
		\end{bmatrix}=g({\bf W})M,$$
		with $g$ a non-negative function and $M\in\mathbb{R}^{2\times 2}$ a semi-positive definite matrix, not necessarily symmetric, then the method is unconditionally stable for any $\theta \in [\frac{1}{2},1]$. 
	\end{corollary}
	\begin{proof}
		Follows immediately from the previous corollary.
	\end{proof}
	
	\begin{corollary}
		If the functions $d^1,d^2,d^3$ and $d^4$ satisfy
		\begin{equation}
			\label{cdts:steptime_semiimplicit_sufficient}
			\begin{aligned}
				&d^1(x,y,t)\geq \tfrac{1}{2}|d^2(x,y,t)+d^3(x,y,t)|,\quad \forall (x,y,t)\in \mathbb{R}^2\times\mathbb{R^+},\\
				&d^4(x,y,t)\geq \tfrac{1}{2}|d^2(x,y,t)+d^3(x,y,t)|,\quad \forall (x,y,t)\in \mathbb{R}^2\times\mathbb{R^+},
			\end{aligned}
		\end{equation}
		then the scheme (\ref{finite_difference_scheme}) is unconditionally stable for any $\theta\in [\frac{1}{2},1]$. 
	\end{corollary}
	\begin{proof}
		Follows from the first corollary by realizing that a matrix $A$ is positive semi-definite in $\mathbb{R}^n$ if and only if its symmetric part $\frac{A+A^\top}{2}$ is positive semi-definite, and that conditions (\ref{cdts:steptime_semiimplicit_sufficient}) force the symmetric part to be diagonally dominant by rows, which is a sufficient condition for positive semi-definiteness \cite{Quarteroni07}.
	\end{proof}

	\begin{corollary}
		If the cross-diffusion matrix (\ref{cross-diffusion matrix}) is of the form 
		$$D(u,v,t)=\begin{bmatrix}
			g(u,v,t) & -f(u,v,t)\\
			f(u,v,t) & g(u,v,t)
		\end{bmatrix},$$
	for some real valued function $f$ and positive real valued function $g$, and $\Delta t$ is such that $\frac{4\Delta t}{h_k^2}\max \frac{g^2+f^2}{g}<1$, then the scheme \ref{finite_difference_scheme} is conditionally stable for the case $\theta=0$.
	\end{corollary}
	\begin{proof}
		Replace $d^1$ and $d^4$ by $g$, $d^2$ by $-f$ and $d^3$ by $f$, and consider $\eta_{1k}=\eta_{2k}=:\eta$ in (\ref{cdts:functions_explicit}) to obtain
		\begin{equation*}
					\begin{aligned}
						\sum\limits_{k=1}^2  (g\delta_k U, \delta_k U)_{h_k^*}+(g\delta_k V, \delta_k V)_{h_k^*}&-\frac{4\Delta t(1+\eta)}{h_k^2}\big(\|g\delta_{k}U\|_{h_k^*}^2+\|f\delta_{k}V\|_{h_k^*}^2+\|f\delta_{k}U\|_{h_k^*}^2+\|g\delta_{k}V\|_{h_k^*}^2 \big)\\
						&\geq\sum\limits_{k=1}^2  \|g^{1/2}\delta_k {\bf W}\|_{h_k^*}^2-\frac{4\Delta t(1+\eta)}{h_k^2}\max \frac{g^2+f^2}{g}\|g^{1/2}\delta_k{\bf W}\|_{h_k^*}^2\\
						&=\sum\limits_{k=1}^2  \big( 1-\frac{4\Delta t(1+\eta)(g^2+f^2)}{h_k^2g}\big)\|g^{1/2}\delta_k{\bf W}\|_{h_k^*}^2\\
						&\geq 0,
					\end{aligned}
		\end{equation*} 
	choosing $\eta$ small enough.
	\end{proof}

	\section{Operator splittings}
	\label{section:splittings}

	The stability conditions suggest that there is a class of functions for which explicit implementations of cross-diffusion processes require unpractical small time steps. However, implicit methods are numerically expensive to solve due to the dimension of the spatial variable. As mentioned before, we will focus on Additive Operator Schemes and Additive Multiplicative Operator Schemes.
	
	\subsection{Splitting models and $L^2$ stability}
	
	Applying the AOS technique \cite{Weickert98} to our finite difference scheme (\ref{finite_difference_scheme}) translates into splitting the spatial operators in each direction, calculate each directional time step and then average the resulting fractional steps. We shall henceforth refer to it as the AOS-CD scheme:  
	
	\begin{subequations}
		\label{aos_scheme}
		\begin{equation}
			\label{aos_fractionalstep}
			\begin{cases}
				\frac{U_{\bf j}^{m+1,k}-U_{\bf j}^{m}}{2\Delta t} =
				 \delta_{k}\left(d^1({\bf W}^m)_{\bf j}^{m+\theta}\delta_{k}U_{{\bf j}}^{m+\theta}+d^2({\bf W}^m)_{\bf j}^{m+\theta}\delta_{k}V_{{\bf j}}^{m+\theta}\right) -\frac{1}{2}\lambda_1({\bf W}^m)_{\bf j}^{m+\theta} (U_{\bf j}^{m+\theta}-U_{\bf j}^0)\\
				\frac{V_{\bf j}^{m+1,k}-V_{\bf j}^{m}}{2\Delta t} =
				 \delta_{k}\left(d^3({\bf W}^m)_{\bf j}^{m+\theta}\delta_{k}U_{{\bf j}}^{m+\theta}+d^4({\bf W}^m)_{\bf j}^{m+\theta}\delta_{k}V_{{\bf j}}^{m+\theta}\right)-\frac{1}{2}\lambda_2({\bf W}^m)_{\bf j}^{m+\theta} (V_{\bf j}^{m+\theta}-V_{\bf j}^0)
			\end{cases} 
		\end{equation}
		for $k=1,2$, and
		\begin{equation}
			\label{aos_averagestep}
			\qquad U_{\bf j}^{m+1}=\frac{1}{2}\sum_{k=1}^{2}U_{\bf j}^{m+1,k},\quad 
			V_{\bf j}^{m+1}=\frac{1}{2}\sum_{k=1}^{2}V_{\bf j}^{m+1,k}.
		\end{equation}
	\end{subequations}
	
	Before moving on to the numerical considerations for this new AOS-CD scheme, we introduce the other splitting technique that we will use to fasten the implicit part of (\ref{finite_difference_scheme}). The motivation is the fact that despite their efficiency and stability, AOS schemes have limited accuracy. An application of AMOS technique \cite{Bara01} to our problem leads to the following AMOS-CD schemes
	
	\begin{subequations}
		\label{amos_scheme}
		\begin{equation}
			\label{amos_factorizationstep1}
			\begin{cases}
				\begin{aligned}
					&\frac{U_{\bf j}^{m+1,*}-U_{\bf j}^{m}}{\Delta t} =
					\delta_{1}\left(d^1({\bf W}^m)_{\bf j}^{m+\theta}\delta_{1}U_{{\bf j}}^{m+\theta}+d^2({\bf W}^m)_{\bf j}^{m+\theta}\delta_{1}V_{{\bf j}}^{m+\theta}\right)-\frac{1}{2}\lambda_1({\bf W}^m)_{\bf j}^{m+\theta} (U_{\bf j}^{m+\theta}-U_{\bf j}^0) \\
					&\frac{V_{\bf j}^{m+1,*}-V_{\bf j}^{m}}{\Delta t} =
					\delta_{1}\left(d^3({\bf W}^m)_{\bf j}^{m+\theta}\delta_{1}U_{{\bf j}}^{m+\theta}+d^4({\bf W}^m)_{\bf j}^{m+\theta}\delta_{1}V_{{\bf j}}^{m+\theta}\right)-\frac{1}{2}\lambda_2({\bf W}^m)_{\bf j}^{m+\theta} (V_{\bf j}^{m+\theta}-V_{\bf j}^0)\\ \\
					&\frac{U_{\bf j}^{m+1,\star}-U_{\bf j}^{m}}{\Delta t} =
					\delta_{2}\left(d^1({\bf W}^m)_{\bf j}^{m+\theta}\delta_{2}U_{{\bf j}}^{m+\theta}+d^2({\bf W}^m)_{\bf j}^{m+\theta}\delta_{2}V_{{\bf j}}^{m+\theta}\right)-\frac{1}{2}\lambda_1({\bf W}^m)_{\bf j}^{m+\theta} (U_{\bf j}^{m+\theta}-U_{\bf j}^0)\\
					&\frac{V_{\bf j}^{m+1,\star}-V_{\bf j}^{m}}{\Delta t} =
					\delta_{2}\left(d^3({\bf W}^m)_{\bf j}^{m+\theta}\delta_{2}U_{{\bf j}}^{m+\theta}+d^4({\bf W}^m)_{\bf j}^{m+\theta}\delta_{2}V_{{\bf j}}^{m+\theta}\right)-\frac{1}{2}\lambda_2({\bf W}^m)_{\bf j}^{m+\theta} (V_{\bf j}^{m+\theta}-V_{\bf j}^0)
				\end{aligned}
			\end{cases}
		\end{equation}
		for the first step,
		\begin{equation}
			\label{amos_factorizationstep2}
			\begin{cases}
				\begin{aligned}
					&\frac{U_{\bf j}^{m+1,**}-U_{\bf j}^{m+1,*}}{\Delta t} =
					\delta_{2}\left(d^1({\bf W}^m)_{\bf j}^{m+\theta}\delta_{2}U_{{\bf j}}^{m+\theta,*}+d^2({\bf W}^m)_{\bf j}^{m+\theta}\delta_{2}V_{{\bf j}}^{m+\theta,*}\right)\frac{1}{2}\lambda_1({\bf W}^m)_{\bf j}^{m+\theta} (U_{\bf j}^{m+\theta,*}-U_{\bf j}^0)\\
					&\frac{V_{\bf j}^{m+1,**}-V_{\bf j}^{m+1,*}}{\Delta t} =
					\delta_{2}\left(d^3({\bf W}^m)_{\bf j}^{m+\theta}\delta_{2}U_{{\bf j}}^{m+\theta,*}+d^4({\bf W}^m)_{\bf j}^{m+\theta}\delta_{2}V_{{\bf j}}^{m+\theta,*}\right)-\frac{1}{2}\lambda_2({\bf W}^m)_{\bf j}^{m+\theta} (V_{\bf j}^{m+\theta,*}-V_{\bf j}^0)\\ \\
					&\frac{U_{\bf j}^{m+1,\star\star}-U_{\bf j}^{m+1,\star}}{\Delta t} =
					\delta_{1}\left(d^1({\bf W}^m)_{\bf j}^{m+\theta}\delta_{1}U_{{\bf j}}^{m+\theta,\star}+d^2({\bf W}^m)_{\bf j}^{m+\theta}\delta_{1}V_{{\bf j}}^{m+\theta,\star}\right)-\frac{1}{2}\lambda_1({\bf W}^m)_{\bf j}^{m+\theta} (U_{\bf j}^{m+\theta,\star}-U_{\bf j}^0) \\
					&\frac{V_{\bf j}^{m+1,\star\star}-V_{\bf j}^{m+1,\star}}{\Delta t} =
					\delta_{1}\left(d^3({\bf W}^m)_{\bf j}^{m+\theta}\delta_{1}U_{{\bf j}}^{m+\theta,\star}+d^4({\bf W}^m)_{\bf j}^{m+\theta}\delta_{1}V_{{\bf j}}^{m+\theta,\star}\right)-\frac{1}{2}\lambda_2({\bf W}^m)_{\bf j}^{m+\theta} (V_{\bf j}^{m+\theta,\star}-V_{\bf j}^0)
				\end{aligned}
			\end{cases}
		\end{equation}
		for the second step, and finally
		\begin{equation}
			\label{amos_average}
			U_{\bf j}^{m+1}=\frac{1}{2}(U_{\bf j}^{m+1,**}+U_{\bf j}^{m+1,\star\star}),\qquad
			V_{\bf j}^{m+1}=\frac{1}{2}(V_{\bf j}^{m+1,**}+V_{\bf j}^{m+1,\star\star}).
		\end{equation}
	\end{subequations}
	Just a quick remark regarding the factorization steps (\ref{amos_factorizationstep1}) and (\ref{amos_factorizationstep2}). In each bracket there are two independent systems of difference equations (in fact there are several of them, as we will discuss in the next section). We chose to display them in this way in order to clarify that systems (\ref{amos_factorizationstep2}) require the solutions of systems (\ref{amos_factorizationstep1}).
	
	In the next result we investigate the numerical stability of both schemes:
	
	\begin{theorem}
		\label{thrm:split_l2_stab}
		If the cross-diffusion matrix (\ref{cross-diffusion matrix}) is such that, for any $U,V\in\mathbb{R}^{N_1\times N_2}$,
		\begin{equation}\label{cdts:aos_imp_stab}
			\begin{aligned}
				(d^1\delta_1 U, \delta_1 U)_{h_1^*}+(d^2\delta_1 V, \delta_1 U)_{h_1^*}+(d^3\delta_1 U, \delta_1 V)_{h_1^*}+(d^4\delta_1 V, \delta_1 V)_{h_1^*}\geq 0,
			\end{aligned}
		\end{equation}
		then the schemes (\ref{aos_scheme}) and (\ref{amos_scheme}) are unconditionally stable for $\theta\in[\frac{1}{2},1]$.
		
		If the cross-diffusion matrix (\ref{cross-diffusion matrix}) is such that, for any $U,V\in\mathbb{R}^{N_1\times N_2}$,
		\begin{equation}
			\begin{aligned}
				\label{cdts:aos_exp_stab}
				&(d^1\delta_1 U, \delta_1 U)_{h_1^*}+(d^2\delta_1 V, \delta_1 U)_{h_1^*}+(d^3\delta_1 U, \delta_1 V)_{h_1^*}+(d^4\delta_1 V, \delta_1 V)_{h_1^*}\\
				&-\Big( \frac{4\Delta t}{h_1^2}\big((1+\eta_{1})(\|d^1\delta_1U\|_{h_1^*}^2+\|d^2\delta_1V\|_{h_1^*}^2 +2(d^1\delta_1U,d^2\delta_1V)_{h_1^*})\\
				&+(1+\eta_{2})(\|d^3\delta_1U\|_{h_1^*}^2+\|d^4\delta_1V\|_{h_1^*}^2+2(d^3\delta_1U,d^4\delta_1V)_{h_1^*}) \big)\Big) \geq 0,
			\end{aligned} 
		\end{equation}
		for some $\eta_1,\eta_2 >0$, and for $k=1,2$, then the schemes (\ref{aos_scheme}) and (\ref{amos_scheme}) are unconditionally stable for the case $\theta=0.$
	\end{theorem}
	\begin{proof}
		We start with the AOS-CD scheme, and proceed as in Theorem \ref{thrm:L2_stab_conditions} for each direction, taking into account that the action in direction $2$ on $U$ and $V$ is the same as the action in direction $1$ on $U^\top$ and $V^\top$ and the triangle inequality to obtain, for $\theta\in[\frac{1}{2},1]$,
		\begin{equation*}
			\|{\bf W}^{m+1}\|_h^2\leq e^{4(\theta^2+(1-\theta)^2)\tilde{\epsilon}^{-1}\epsilon t^{m+1}}\big(1+2\tilde{\epsilon}^{-1}t^{m+1}\lambda_\text{max}\big)\|{\bf W^0}\|_h^2,
		\end{equation*}
		with $\tilde{\epsilon}<1-4\Delta t \epsilon \theta^2$ for some $\epsilon>0$. For the case $\theta=0$ the same steps as in the proof of Theorem \ref{thrm:L2_stab_conditions} yield
		\begin{equation}
			\begin{aligned}
				&\|{\bf W}^{m+1}\|_h^2\leq  e^{a_\epsilon t^{m+1}}\big(1+t^{m+1}b_\epsilon\big)\|{\bf W}^0\|_h^2,
			\end{aligned}
		\end{equation}
		with $a_\epsilon= 4\big(\epsilon+\frac{\zeta\Delta t}{2}\big)$ and $b_\epsilon= 4\big( \frac{\lambda_\text{max}^2}{4\epsilon} +\frac{\zeta\Delta t}{2} \big)$, where $\zeta=\max \{1+\eta_1^{-1},1+\eta_2^{-1}\}\lambda_\text{max}^2$ and $\epsilon >0$.
		For the AMOS-CD scheme, follow the same steps to get similar inequalities for each factorization step (\ref{amos_factorizationstep1}) and (\ref{amos_factorizationstep2}) and finish with the triangle inequality in (\ref{amos_average}). For the case $\theta\in [\frac{1}{2},1]$ we obtain
		\begin{equation*}
			\begin{aligned}
				\|{\bf W}^{m+1}\|_h^2 \leq e^{8(\theta ^2 + (1-\theta)^2)\tilde{\epsilon}^{-1}\epsilon t^{m+1}}\big(1+t^{m+1}\lambda_\text{max}\epsilon^{-1}\tilde{\epsilon}^{-1} \big)\|{\bf W}^0\|_h^2,
			\end{aligned}
		\end{equation*}
		with $\tilde{\epsilon}<1-2\Delta t \epsilon\theta^2$ for some $\epsilon>0$, while for the case $\theta=0$ we obtain 
		\begin{equation}
			\begin{aligned}
				&\|{\bf W}^{m+1}\|_h^2\leq  e^{a_\epsilon t^{m+1}}\big(1+t^{m+1}b_\epsilon\big)\|{\bf W}^0\|_h^2,
			\end{aligned}
		\end{equation}
		with $a_\epsilon= 4(\epsilon+\zeta\Delta t)(1+\Delta t)$ and $b_\epsilon= 2\big( \frac{\lambda_\text{max}}{4\epsilon} +\frac{\zeta\Delta t}{2} \big)\big( 2+2\Delta t\big(\epsilon+\zeta\Delta t \big)\big).$
	\end{proof}	
	
	\subsection{Computational considerations}
	
	We now discuss the computational implementations for each scheme. Let us start with the matrix formulation for the $\theta$-method (\ref{finite_difference_scheme}) obtained by arranging the elements of $(U,V)$ in a vector defining
	\begin{equation*}
		w:={ \begin{bmatrix}
				U_{1,1}&
				U_{2,1}&
				U_{3,1}&
				\cdots&
				U_{N_1,N_2}&
				V_{1,1}&
				V_{2,1}&
				V_{3,1}&
				\cdots&
				V_{N_1,N_2}
		\end{bmatrix}}^\top,
	\end{equation*}
	in what we call an \textit{ordering of pixels prioritizing the first direction} (order first in direction $1$ and second in direction $2$). As such, we set
	\begin{equation}
		\begin{aligned}
			&A_{11}^s:=\sum_{\vartheta_k=1}^{2}\vartheta_k^l D_k^1(w^m)^s \vartheta_k^r,\qquad A_{12}^s:=\sum_{\vartheta_k=1}^{2}\vartheta_k^l D_k^2(w^m)^s \vartheta_k^r ,\\
			&A_{22}^s:=\sum_{\vartheta_k=1}^{2}\vartheta_k^l D_k^4(w^m)^s \vartheta_k^r ,\qquad A_{21}^s:=\sum_{\vartheta_k=1}^{2}\vartheta_k^l D_k^3(w^m)^s \vartheta_k^r,
		\end{aligned}
		\label{submatricesA}
	\end{equation}
	for $s=n,n+1$, where $\vartheta_k^r$ denote the backward difference operators with respect to direction $k$, $\vartheta_k^l$ denote the forward difference operators with respect to direction $k$. $D_k^\ell(w^m)^s$ is a diagonal matrix with $[D_k^\ell(w^m)^s]_{j,j}:=d^\ell(w^m)_{{\bf j}-e_k/2}^s$, for $\ell=1,\dots,4$, following the backward average (\ref{definition_D}) in the corresponding direction $k$, and
	\begin{equation}
		\label{matrixA}
		A^s:={ \begin{bmatrix}
				A_{11}^s & A_{12}^s\\
				A_{21}^s & A_{22}^s
		\end{bmatrix}},
	\end{equation}
	and
	\begin{equation*}
		\begin{aligned}
			&\Lambda_i(w^m)^s:=\text{diag}\{\lambda_i(w^m)_{1,1}^s,\lambda_i(w^m)_{2,1}^s,\dots,\lambda_i(w^m)_{N_1,N_2}^s\},\\
			&\Lambda^s:=\text{diag}\{\Lambda_1(v^m)^s,\Lambda_2(v^m)^s\},
		\end{aligned}
	\end{equation*}
	The $\theta$-method iteration $m+1$ in (\ref{finite_difference_scheme}) is 
	\begin{equation}\label{linearsystem:full}
		\begin{aligned}
			\bigg(I-\theta\Delta t A^{m+1}-\theta\Delta t\Lambda^{m+1}\bigg)w^{m+1}=&\bigg(I+(1-\theta)\Delta t A^m+(1-\theta)\theta\Delta t\Lambda^m\bigg)w^m\\
			&+\Delta t(\theta \Lambda^{m+1}+(1-\theta)\Lambda^{m}) w^0,
		\end{aligned}
	\end{equation}
	for $m\in \mathbb{N}_0$. 
	
	On the other hand, for each of the factorization steps (\ref{aos_fractionalstep}), (\ref{amos_factorizationstep1}) and (\ref{amos_factorizationstep2}), the iteration is	
	\begin{equation}\label{linearsystem:splitted}
		\bigg(I-\theta r A_k^{m+1}+\theta r \Lambda_k^{m+1}\bigg)w_k^{m+1}=\bigg(I+(1-\theta) r A_k^m-(1-\theta) r   \Lambda_k^m\bigg)w_k^m+ r (\theta \Lambda_k^{m+1}+(1-\theta)\Lambda_k^m) w_k^0,
	\end{equation}
	where $w_k^s$ is vector $w^s$ with an orientation prioritizing direction $k$ and $\Lambda_k^s$ follows the same orientation. Here, $r$ is defined as 
	\begin{equation}
		\label{r_aos}
		r:=2\theta\Delta t
	\end{equation}
	for the AOS-CD splitting scheme, and 
	\begin{equation}
		\label{r_amos}
		r:=\theta\Delta t
	\end{equation}
	for the AMOS-CD scheme. Matrix $A_k^s$ is defined as 
	\begin{equation*}
		A_k^s:=\begin{bmatrix}
				\vartheta_k^l D_k^1(w^m)^s \vartheta_k^r & \vartheta_k^l D_k^2(w^m)^s \vartheta_k^r\\
				\vartheta_k^l D_k^3(w^m)^s \vartheta_k^r & \vartheta_k^l D_k^4(w^m)^s \vartheta_k^r
		\end{bmatrix}.
	\end{equation*}
	
	It is now verifiable the increase on efficiency when we resort to a splitting scheme. Instead of solving a full system with $2N_1N_2$ equations and unknowns in (\ref{linearsystem:full}), we now need to solve $N_\ell$ independent systems with $2N_k$ equations and unknowns in (\ref{linearsystem:splitted})), where $\ell$ is the direction orthogonal to $k$, for each factorization step (we will explore the actual speedups in Section \ref{section:speedups}). For a visual representation of this shift see Figure \ref{fig:splitting_visualization}.
	\begin{figure*}[!htb]
		\center
		\minipage{0.45\textwidth}
		\includegraphics[width=\linewidth]{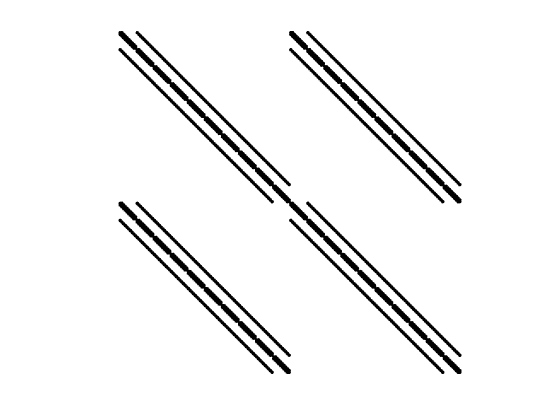}
		\endminipage
		\minipage{0.45\textwidth}
		\includegraphics[width=\linewidth]{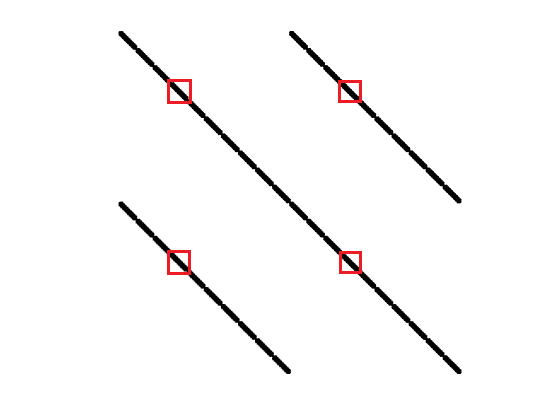}
		\endminipage
		\caption{Splitting techniques effect on the two dimensional cross-diffusion process. Left is the system matrix in the full implicit difference scheme and right is the full system matrix of the splitting technique (squares highlight an independent system of equations).}
		\label{fig:splitting_visualization}
	\end{figure*}
	Although substantial gains are achieved by this fact alone, the true strength of splitting techniques lies in the possibility of using fast forward and backward substitution system solvers such as the Thomas algorithm. To obtain an analogous implementation for the cross-diffusion case, we can reorder the unknowns in $w$ by interchanging the pixels in $U$ and $V$ in the direction to be diffused. That is, if we are solving system $i$ in direction $1$, we order the pixels as 
	\begin{equation}\label{w_ordered_blockTridiagonal}
		w={ \begin{bmatrix}
				U_{1,i}&
				V_{1,i}&
				U_{2,i}&
				V_{2,i}&
				\cdots&
				U_{N_1,i}&
				V_{N_1,i}
		\end{bmatrix}}^\top,
	\end{equation}
	while for solving system $j$ in direction $2$ we order the pixels as
	\begin{equation}\label{w_ordered_2_blockTridiagonal}
		w={ \begin{bmatrix}
				U_{j,1}&
				V_{j,1}&
				U_{j,2}&
				V_{j,2}&
				\cdots&
				U_{j,N_2}&
				V_{j,N_2}
		\end{bmatrix}}^\top,
	\end{equation}
	The system matrix becomes, at each iteration $m$,
	\begin{equation}\label{matrix:blockTridiagonal}
		A={ \begin{bmatrix}
				B_{1,i}^{m+\theta} & U_{1,i}^{m+\theta} & & & & \\
				L_{1,i}^{m+\theta} & B_{2,i}^{m+\theta} & U_{2,i}^{m+\theta} & & & \\
				& L_{2,i}^{m+\theta} & B_{3,i}^{m+\theta} & \ddots &  & \\
				& & \ddots & \ddots & \ddots & \\
				& & & \ddots & \ddots & U_{N_1-1,i}^{m+\theta} \\
				& & & & L_{N_1-1,i}^{m+\theta} & B_{N_1,i}^{m+\theta}
		\end{bmatrix}},
	\end{equation}
	with
	\begin{equation*}
		\begin{aligned}
			&B_{j,i}^{m+\theta}=\begin{bmatrix}
					1+2r d^1({\bf W}^m)^{m+\theta}_{j\pm 1/2,i}+\frac{1}{2}r\lambda_{1j,i}^{m+\theta} & 2r d^2({\bf W}^m)^{m+\theta}_{j\pm 1/2,i}\\
					2r d^3({\bf W}^m)^{m+\theta}_{j\pm 1/2,i} & 1+2r d^4({\bf W}^m)^{m+\theta}_{j\pm 1/2,i}+\frac{1}{2}r\lambda_{2j,i}^{m+\theta}
			\end{bmatrix},\\
			&B_{j,i}^{m+\theta}=\begin{bmatrix}
					1+r (d^1({\bf W}^m)^{m+\theta}_{j-1/2,i}+d^1({\bf W}^m)^{m+\theta}_{j+1/2,i})+\frac{1}{2}r\lambda_{1j,i}^{m+\theta} & r (d^2({\bf W}^m)^{m+\theta}_{j-1/2,i}+d^2({\bf W}^m)^{m+\theta}_{j+1/2,i})\\
					r (d^3({\bf W}^m)^{m+\theta}_{j-1/2,i}+d^3({\bf W}^m)^{m+\theta}_{j+1/2,i}) & \hspace{-1cm} 1+r (d^4({\bf W}^m)^{m+\theta}_{j-1/2,i}+d^4({\bf W}^m)^{m+\theta}_{j+1/2,i})+\frac{1}{2}r\lambda_{2j,i}^{m+\theta}
			\end{bmatrix},\\
			&U_{1,i}^{m+\theta}=\begin{bmatrix}
					-2r d^1({\bf W}^m)^{m+\theta}_{3/2,i} & -2r d^2({\bf W}^m)^{m+\theta}_{3/2,i}\\
					-2r d^3({\bf W}^m)^{m+\theta}_{3/2,i} & -2r d^4({\bf W}^m)^{m+\theta}_{3/2,i}
			\end{bmatrix},\\
			&L_{N-1-1,i}^{m+\theta}=\begin{bmatrix}
					-2r d^1({\bf W}^m)^{m+\theta}_{N_1-1/2,i} & -2r d^2({\bf W}^m)^{m+\theta }_{N_1-1/2,i}\\
					-2r d^3({\bf W}^m)^{m+\theta}_{N_1-1/2,i} & -2r d^4({\bf W}^m)^{m+\theta}_{N_1-1/2,i}
			\end{bmatrix},\\
			&L_{j,i}^{m+\theta}=U_{j,i}^{m+\theta}=\begin{bmatrix}
					-r d^1({\bf W}^m)^{m+\theta}_{j+1/2,i} & -r d^2({\bf W}^m)^{m+\theta}_{j+1/2,i}\\
					-r d^3({\bf W}^m)^{m+\theta}_{j+1/2,i} & -r d^4({\bf W}^m)^{m+\theta}_{j+1/2,i}
			\end{bmatrix},
		\end{aligned}
	\end{equation*}
	where we dropped the unvarying subscripts and superscript to avoid unnecessary cluttering. 
	
	From now on we will assume that the off-diagonal submatrices are not zero. If we have, by any chance, $U_j=L_{j-1}=0$ for some $j$, then all the properties we derive for $A$ can be obtained for the tridiagonal by blocks submatrices found when we separate $A$ along the row and column where those elements vanish.
	
	\subsection{Existence, uniqueness and stability of the block LU factorization}
	\label{section:lu_factorization}
	
	Let us write matrix (\ref{matrix:blockTridiagonal}) in its block $LU$ factorization (if it exists)
	\begin{equation}
		\label{blockLU_factorization}
		LU={\begin{bmatrix}
				\bar B_{1} &  & & & & \\
				L_{1} & \bar B_{2} & & & & \\
				& L_{2} & \bar B_{3} &  &  & \\
				& & \ddots & \ddots &  & \\
				& & & \ddots & \ddots &  \\
				& & & & L_{N_1-1} & \bar B_{N_1}
			\end{bmatrix}\begin{bmatrix}
				I_2 & \bar U_{1} & & & & \\
				& I_2 & \bar U_{2} & & & \\
				&  & I_2 & \ddots &  & \\
				& &  & \ddots & \ddots & \\
				& & &  & \ddots & \bar U_{N_1-1} \\
				& & & &  & I_2
		\end{bmatrix}},
	\end{equation}
	with $I_2$ the $2\times 2$ square identity matrix and
	\begin{equation}
		\label{blockLU_decomposition_steps}
		\begin{aligned}
			&\bar B_{1}=B_{1},\\
			& \bar B_{j}\bar U_{j}=U_{j},\quad j=1,\dots,N_1-1,\\
			& \bar B_{j+1}=B_{j+1}-L_{j}\bar U_{j},\quad j=1,\dots,N_1-1,
		\end{aligned}
	\end{equation}
	
	In the following series of results, we answer the questions of when and how this factorization works. We start by proving the existence and uniqueness of the block $LU$ factorization, whose conditions motivate further results. All the results will be taken on matrix (\ref{matrix:blockTridiagonal}) corresponding to the action on direction $1$, but they extend to other directions without any further ado.
	
	\begin{lemma}
		\label{lemma:blockuniqueexists}
		If the block matrices in (\ref{matrixA}) are such that
		\begin{enumerate}
			\item $B_j$ are nonsingular,
			\item $||B_1^{-1}U_1||<1$,
			\item $||B_j^{-1}U_j||+||B_j^{-1}L_{j-1}||\leq 1$,
		\end{enumerate}
		then the block $LU$ factorization (\ref{blockLU_factorization}) (\ref{blockLU_decomposition_steps}) exists and is unique. Furthermore, the resulting matrices $\bar B_j$ are non singular and satisfy
		\begin{equation*}
			\| \bar B_j ^{-1} U_j\| < 1,
		\end{equation*}
		for all $j=1,\dots,N_1$.
	\end{lemma}
	\begin{proof}
		Consider the first step of the $LU$ factorization in (\ref{blockLU_decomposition_steps}). It leads to the decomposition
		$$A={\begin{bmatrix}
				\bar B_1 & 0\\
				L_1 & S
			\end{bmatrix}\begin{bmatrix}
				I_2 &\bar U_1 \\
				0 & I_{2N_1-2}
		\end{bmatrix}},\quad \text{where }S={\begin{bmatrix}
				\bar B_2 & U_2 & & & \\
				L_2 & B_3 & \ddots & & \\
				& \ddots & \ddots & \ddots & \\
				& & \ddots & \ddots & U_{N_1-1} \\
				& & & L_{N_1-1} & B_{N_1}
		\end{bmatrix}},$$
		and $\bar B_2 =B_2 - L_1\bar B_1^{-1} U_1=B_2(I_2-B_2^{-1}L_1\bar B_1^{-1}U_1)$. By hypothesis 2 and 3 we have that
		$$\|B_2^{-1}L_1\bar B_1^{-1}U_1\|\leq\|B_2^{-1}L_1\|\|\bar B_1^{-1}U_1\|<1$$
		and therefore $(I_2-B_2^{-1}L_1\bar B_1^{-1}U_1)^{-1}$ exists \cite{Dym08}, and 
		$$\|(I_2-B_2^{-1}L_1\bar B_1^{-1}U_1)^{-1}\|\leq \frac{1}{1-\|L_1B_2^{-1}\|\|\bar B_1U_1\|}.$$
		By hypothesis 1, $B_2$ is also nonsingular, so $\bar{B}_2$ is nonsingular and, using hypothesis 3,
		\begin{equation*}
			\begin{aligned}
				\| \bar B_2^{-1} U_2\| &=\| (I_2-B_2^{-1}L_1\bar B_1^{-1}U_1)^{-1}B_2^{-1} U_1\|\\
				&\leq \frac{\| B_2^{-1} U_2\|}{1-\|B_2^{-1}L_1\|\|\bar B_1^{-1}U_1\|}\\
				&\leq \frac{1-\| B_2^{-1} L_1\|}{1-\|B_2^{-1}L_1\|\|\bar B_1^{-1}U_1\|}\\
				&< \frac{1-\| B_2^{-1} L_1\|}{1-\|B_2^{-1}L_1\|}=1.
			\end{aligned}
		\end{equation*}
		As the factorization now proceeds independently in matrix $S$, the result follows by induction.
	\end{proof}
	
	This lemma requires, among other conditions, the non singularity of matrices $B_j$, which can be obtained under fairly nonrestrictive constraints.
	
	\begin{lemma}
		\label{lemma:inv_submatrices}
		If the cross-diffusion matrix (\ref{cross-diffusion matrix}) is positive semi-definite, not necessarily symmetric, and $\lambda_1,\lambda_2>0$, then the diagonal submatrices of matrix (\ref{matrix:blockTridiagonal}) are nonsingular.
	\end{lemma}
	\begin{proof}
		We check the nonsingularity of $B_1$. The proof for the remaining submatrices is exactly the same. Set $k_{1j}=\frac{1}{2r}(1+\frac{1}{2}r\lambda_{1j})$ and $k_{2j}=\frac{1}{2r}(1+\frac{1}{2}r\lambda_{2j})$. Let $\frac{B_1+B_1^\top}{2}$ be the hermitian part of $B_1$. Noticing that
		$$\frac{B_1+B_1^\top}{2}={\begin{bmatrix}
				k_{1j}+r(d^1_j+d^1_{j+1}) & \frac{r}{2}(d^2_j+d^3_j+d^2_{j+1}+d^2_{j+1})\\
				\frac{r}{2}(d^2_j+d^3_j+d^2_{j+1}+d^2_{j+1}) & k_{2j}+r(d^4_j+d^4_{j+1})
		\end{bmatrix}},$$
		and since the cross-diffusion matrix is positive semi-definite not necessarily symmetric, its hermitian part is positive semi-definite and so $\frac{B_1+B_1^\top}{2}$ is positive definite. This implies that all its eigenvalues are positive. But then the real part of the eigenvalues of $B_1$ are positive which implies that $0$ is not an eigenvalue of $B_1$ \cite{Browne30}.
	\end{proof}
	
	We now proceed on checking the remaining conditions of Lemma \ref{lemma:blockuniqueexists}. We require some known results, which we provide in the following lemma. The first two propositions are exercise 12.9 and Lemma 12.29 from Dym \cite{Dym08}. The third proposition is Theorem 7.7.2 by Horn and Johnson \cite{Horn12}.
	
	\begin{lemma}
		\label{lemma:books_props}
		Let $A,S\in \mathbb{C}^{n\times n}$ and $\alpha\in\mathbb{R}^+$. Then
		\begin{enumerate}
			\item $|| A || < 1 \iff I- A^\top A\text{ is positive definite}\iff I-AA^\top\text{ is positive definite}$,
			\item $|| A || < \alpha \iff \alpha^2I- A^\top A\text{ is positive definite}\iff \alpha^2 I-AA^\top\text{ is positive definite}$,
			\item $A\text{ is positive definite}\iff S^\top AS \text{ is positive definite, for any nonsingular }S$.
		\end{enumerate}
	\end{lemma}
	Setting the auxiliary variables $k_{1j}:=1+\frac{1}{2}r\lambda_{1j}$, $k_{2j}:=1+\frac{1}{2}r\lambda_{2j}$ and $K_j:={\begin{bmatrix}
			k_{1j} & 0 \\
			0 & k_{2j}
	\end{bmatrix}}$, we are able to evaluate the second condition of Lemma \ref{lemma:blockuniqueexists}.
	
	\begin{lemma}
		\label{lemma:firstmatrixNormsBounds}
		Consider the matrices $B_1$ and $U_1$ in (\ref{matrix:blockTridiagonal}). If the cross-diffusion matrix (\ref{cross-diffusion matrix}) is positive semi-definite, not necessarily symmetric, and
		\begin{enumerate}
			\item $\lambda_{1}(x)=\lambda_{2}(x)$ for all $x$, then $|| B_1^{-1}U_1||<1$ for all $r\in \mathbb{R}^+$;
			\item $\lambda_{1}(x)>\lambda_{2}(x)$, for all $x$, and the matrix
			\begin{equation}
				\label{lemma4_matrix1}
				\begin{bmatrix}
					k_{1}^2(x)+4r(k_{1}(x)-k_{2}(x))d^1(x) & 2r(k_{1}(x)-k_{2}(x))d^2(x)\\
					2r(k_{1}(x)-k_{2}(x))d^2(x) & k_{2}^2(x)
				\end{bmatrix}
			\end{equation}
			is positive definite for all $x$, then $|| B_1^{-1}U_1||<1$ for all $r\in \mathbb{R}^+$;
			\item $\lambda_{2}(x)>\lambda_{1}(x)$, for all $x$, and the matrix
			\begin{equation}
				\label{lemma4_matrix2}
				\begin{bmatrix}
					k_{1}^2(x) & 2r(k_{2}(x)-k_{1}(x))d^3(x)\\
					2r(k_{2}(x)-k_{1}(x))d^3(x) & k_{2}^2(x)+4r(k_{2}(x)-k_{1}(x))d^4(x)
				\end{bmatrix}
			\end{equation}
			is positive definite for all $x$, then $|| B_1^{-1}U_1||<1$ for all $r\in \mathbb{R}^+$.
			\item $\lambda_{1}(x)-\lambda_{2}(x)$ changes sign and the above two matrices are positive definite for all $x$, then $|| B_1^{-1}U_1||<1$ for all $r\in \mathbb{R}^+$.
		\end{enumerate}
	\end{lemma}
	\begin{proof}
		We will use Lemma \ref{lemma:books_props} thoroughly during the proof. Notice first that that $U_1=K_1-B_1$. Then $B_1^{-1}U_1=B_1^{-1}K_1-I$. We have that $||B_1^{-1}U_1||_2<1$ if and only if $I-B_1^{-1}U_1(B_1^{-1}U_1)^\top$ is positive definite. But
		\begin{equation*}
			\begin{aligned}
				I-B_1^{-1}U_1(B_1^{-1}U_1)^\top&=I-(B_1^{-1}K_1-I)(K_1^\top B_1^{-1\top}-I)\\
				&=B_1^{-1}(B_1K_1+K_1B_1^\top-K_1^2)B_1^{-1\top},
			\end{aligned}
		\end{equation*}
		so the positive definiteness of $B_1K_1+K_1B_1^\top-K_1^2$ provides the result. We can compute
		\begin{equation*}
			\begin{aligned}
				B_1K_1+K_1B_1^\top-K_1^2={\begin{bmatrix}
						k_{11}^2+4rk_{11}d_{3/2}^1 \hspace{-0.5cm}& \hspace{-0.5cm}2rk_{11}d_{3/2}^2+2rk_{21}d_{3/2}^3\\
						2rk_{11}d_{3/2}^2+2rk_{21}d_{3/2}^3 \hspace{-0.5cm}& \hspace{-0.5cm}k_{21}^2+4rk_{21}d_{3/2}^4
				\end{bmatrix}}.
			\end{aligned}
		\end{equation*}
		If $\lambda_{11}=\lambda_{21}$, $BK_1+K_1B^\top-K_1^2$ is positive definite due to the positive semi-definiteness of the hermitian part of the cross-diffusion matrix. If $\lambda_{11}>\lambda_{21}$, then $k_{11}>k_{21}$ and we can write
		\begin{equation*}
			\begin{aligned}
				BK_1+K_1B^\top-K_1^2=&2rk_{21}{\begin{bmatrix}
						2d_{3/2}^1 & d_{3/2}^2+d_{3/2}^3\\
						d_{3/2}^2+d_{3/2}^3 & 2d_{3/2}^4
				\end{bmatrix}}\\
				&\quad +
				{\begin{bmatrix}
						k_{11}^2+4r(k_{11}-k_{21})d_{3/2}^1 & 2r(k_{11}-k_{21})d_{3/2}^2 \\
						2r(k_{11}-k_{21})d_{3/2}^2 & k_{21}^2
				\end{bmatrix}},
			\end{aligned}
		\end{equation*}
		with the first summand being positive semi-definite due to the positive semi-definiteness of the hermitian part of the cross diffusion matrix. Therefore the result is obtained if the second summand is also positive definite, which can be achieved by splitting that matrix using $d^\ell_{3/2}=\frac{1}{2}(d^\ell_{1}+d^\ell_2)$ and taking into account the hypothesis.
		
		The proof for $\lambda_{21}>\lambda_{11}$ is analogous.
	\end{proof}
	
	Further restrictions on the cross-diffusion matrix are required for the third condition of Lemma \ref{lemma:blockuniqueexists}. In the following result we check that the inequality $\|B_j^{-1} U_j\|_2+\|B_j^{-1} L_{j-1}\|_2\leq 1$ holds unconditionally for a special choice of the cross-diffusion matrix (\ref{cross-diffusion matrix}), and that it holds for the general case of positive semi-definite, not necessarily symmetric, cross-diffusion matrix, as long as we take small enough time steps.
	
	\begin{lemma}
		\label{lemma:midmatricesNormsBounds}
		If the functions $\lambda_1(\cdot),$ $\lambda_2(\cdot)$ and the cross-diffusion matrix (\ref{cross-diffusion matrix}) are such that
		\begin{enumerate}
			\item $\lambda_{1}(x)=\lambda_{2}(x)$ for all $x$, or
			\item $\lambda_{1}(x)>\lambda_{2}(x)$, for all $x$, and
			\begin{equation}
				\label{lemma5_matrix1}
				\begin{bmatrix}
					k_{1}^2(x)+2r(k_{1}(x)-k_{2}(x))d^1(x) & r(k_{1}(x)-k_{2}(x))d^2(x)\\
					r(k_{1}(x)-k_{2}(x))d^2(x) & k_{2}^2(x)
				\end{bmatrix}
			\end{equation}
			is positive semi-definite for all $x$, or
			\item $\lambda_{2}(x)>\lambda_{1}(x)$, for all $x$, and
			\begin{equation}
				\label{lemma5_matrix2}
				\begin{bmatrix}
					k_{1}^2(x) & r(k_{2}(x)-k_{1}(x))d^3(x)\\
					r(k_{2}(x)-k_{1}(x))d^3(x) & k_{2}^2(x)+2r(k_{2}(x)-k_{1}(x))d^4(x)
				\end{bmatrix}
			\end{equation}
			is positive semi-definite for all $x$, or
			\item $\lambda_{1}(x)-\lambda_{2}(x)$ changes sign, and the above two matrices are positive semi-definites for all $x$,
		\end{enumerate}	
		then $\|B_j^{-1}U_j\|_2+\|B_j^{-1} L_{j-1}\|_2\leq 1$ for any choice of $r$ in (\ref{r_aos})-(\ref{r_amos}) for cross-diffusion matrices of the form $g(\cdot)M$ for some non-negative real valued function $g$ and positive semi-definite matrix $M$.
		Furthermore, for a general positive semi-definite, not necessarily symmetric, cross-diffusion matrix, it holds that $\|B_j^{-1}U_j\|_2+\|B_j^{-1} L_{j-1}\|_2\leq 1$ for sufficiently small $r$ as long as we change the requirements on the matrices in conditions 2-4 to positive definiteness.
	\end{lemma}
	\begin{proof}
		We will use the notation $A\succeq 0$ to indicate that $A$ is positive semi-definite. Suppose that the cross-diffusion matrix is of the form $g(\cdot)M$. Then  $B_j=K_j^2+r(g_{j-1/2}+g_{j+1/2})MM^\top$ and
		\begin{equation*}
			\begin{aligned}
				\|B_j^{-1}U_j\|\leq \alpha &\iff \alpha^2 B_jB_j^\top - U_jU_j^\top \succeq 0\\
				&\iff \alpha^2 ( K_j^2 + r(g_{j-1/2}+g_{j+1/2})(MK_j+K_jM^\top) \\
				&\qquad+ r^2(g_{j-1/2}+g_{j+1/2})^2MM^\top ) - r^2g_{j+1/2}^2 MM^\top\succeq 0\\
				&\impliedby \alpha^2r^2(g_{j-1/2}+g_{j+1/2})^2MM^\top - r^2g_{j+1/2}^2 MM^\top\succeq 0 \\
				&\impliedby \alpha^2(g_{j-1/2}+g_{j+1/2})^2-g_{j+1/2}^2\geq 0 \\
				&\iff \alpha \geq \frac{g_{j+1/2}}{g_{j-1/2}+g_{j+1/2}},
			\end{aligned}
		\end{equation*}
		where we have used the fact that, considering the notation $M=[m_{ij}]$,
		\begin{equation*}
			\begin{aligned}
				K_j^2 + r(g_{j-1/2}&+g_{j+1/2})(MK_j+K_jM^\top)\\
				&=K_j^2+r(g_{j-1/2}+g_{j+1/2}){\begin{bmatrix}
						2k_{1j}m_{11} & k_{1j}m_{12}+k_{2j}m_{12}\\
						k_{1j}m_{21}+k_{2j}m_{21} & 2k_{1j}m_{22}
				\end{bmatrix}}	
			\end{aligned}
		\end{equation*}
		is positive semi-definite by the set of hypothesis 1-4 (the details are precisely the same as in the previous lemma). In the same way,
		\begin{equation*}
			\beta\geq \frac{g_{j-1/2}}{g_{j-1/2}+g_{j+1/2}}\implies \|B_j^{-1}L_{j-1}\|\leq \beta,
		\end{equation*}
		and so we choose in particular $\alpha=\frac{g_{j+1/2}}{g_{j-1/2}+g_{j+1/2}}$ and $\beta=\frac{g_{j-1/2}}{g_{j-1/2}+g_{j+1/2}}$ to obtain
		$\|B_j^{-1}U_j\|+\|B_j^{-1}L_{j-1}\|\leq \alpha + \beta = 1$.
		
		In order to prove the second part of the lemma, notice that for the general case of positive semi-definite cross-diffusion matrix,
		\begin{equation*}
			\begin{aligned}
				\|B_j^{-1}U_j\|\leq \alpha&\iff \alpha^2(K_jB_j^\top + B_jK_j - K_j^2 \\
				&\qquad+ L_{j-1}U_j^\top+U_jL_{j-1}^\top+L_{j-1}L_{j-1}^\top+U_jU_j^\top) -U_jU_j^\top\succeq 0\\
				&\iff \alpha^2( K_jB_j^\top + B_jK_j - K_j^2 ) \\
				&\qquad+ \alpha^2(L_{j-1}U_j^\top+U_jL_{j-1}^\top+L_{j-1}L_{j-1}^\top)+(\alpha^2-1)U_jU_j^\top \succeq 0.
			\end{aligned}
		\end{equation*}
		In the same way, it holds that 
		\begin{equation*}
			\begin{aligned}
				\|B_j^{-1}L_{j-1}\|\leq \beta &\iff \beta^2( K_jB_j^\top + B_jK_j - K_j^2 ) \\
				&\quad+ \beta^2(L_{j-1}U_j^\top+U_jL_{j-1}^\top+U_jU_j^\top)+(\beta^2-1)L_{j-1}L_{j-1}^\top \succeq 0.
			\end{aligned}
		\end{equation*}
		By hypothesis, $K_jB_j^\top + B_jK_j - K_j^2$ is positive definite. Notice further that although the indefiniteness of this term occurs with increasing values of $r$, in the limit $r=0$ this term is a positive diagonal matrix. Furthermore, matrices $L_{j-1}U_j^\top, U_jL^\top, L_{j-1}L_{j-1}^\top$ and $U_jU_j^\top$ are scaled by $r^2$, and therefore
		\begin{equation*}
			\begin{aligned}
				&\|B_j^{-1}U_j\|\leq \alpha \iff \alpha^2(K_jB_j^\top + B_jK_j - K_j^2 ) + \mathcal{O}(r^2) \succeq 0,\\
				&\|B_j^{-1}L_{j-1}\|\leq \beta \iff \beta^2(K_jB_j^\top + B_jK_j - K_j^2 ) + \mathcal{O}(r^2) \succeq 0,
			\end{aligned}
		\end{equation*}
		so we can always find $r$ such that the right hand sides are positive semi-definite and $\alpha+\beta\leq 1$.
	\end{proof}
	
	Taking into account the previous lemmas, we have all the ingredients to formulate a theorem on existence and uniqueness of the block LU factorization of matrix (\ref{matrix:blockTridiagonal}). We further enhance this result by discussing the stability of the factorization. The cornerstone of our stability approach is a result by Demmel, Higham and Schreiber \cite{Higham02}:
	
	\begin{theorem}
		Let $\hat{L}$ and $\hat{U}$ be the computed block $LU$ factors of $A\in\mathbb{R}^{2N_1 \times 2N_1}$ from algorithm (\ref{blockLU_decomposition_steps}), and let $\hat{x}$ be the computed solution to $Ax=b$ under standard underlying level-3 BLAS (matrix-matrix operations) assumptions (for details, see e.g. \cite{Golub13,Higham02}). Then
		\begin{equation*}
			\begin{aligned}
				&\hat{L}\hat{U}=A+\Delta A_1,\quad (A+\Delta A_2)\hat{x}=b,\\
				&\|\Delta A_i\|\leq d_{N_1}u(\|A\|+\|\hat{L}\|\|\hat{U}\|)+\mathcal{O}(u^2),\quad i=1,2,
			\end{aligned}
		\end{equation*}
		where $d_{N_1}$ is a constant commensurate with those in the assumptions and $u$ is the machine unit roundoff.
	\end{theorem}
	
	As pointed in \cite{Higham02}, the stability of block $LU$ factorization lies in the ratio $\frac{\|\hat{L}\|\|\hat{U}\|}{\|A\|}$. If it is bounded by a reasonable function of $N_1$, then we can say that the computed solution solves a slightly perturbed version of the original system. We make the usual assumption that $\|L\|\|\|U\|\approx \|\hat{L}\|\|\|\hat{U}\|$ to attain a functional bound in the concluding theorem of this section.
	
	\begin{theorem}
		\label{thrm:stab_final}
		If the functions $\lambda_1(\cdot),$ $\lambda_2(\cdot)$ and the cross-diffusion matrix (\ref{cross-diffusion matrix}) are such that
		\begin{enumerate}
			\item $\lambda_{1}(x)=\lambda_{2}(x)$ for all $x$, or
			\item if $\lambda_{1}(x)>\lambda_{2}(x)$, for all $x$, and
			$${\begin{bmatrix}
					k_{1}^2(x)+4r(k_{1}(x)-k_{2}(x))d^1(x) & 2r(k_{1}(x)-k_{2}(x))d^2(x)\\
					2r(k_{1}(x)-k_{2}(x))d^2(x) & k_{2}^2(x)
			\end{bmatrix}}$$
			is positive semi-definite for all $x$, or
			\item if $\lambda_{2}(x)>\lambda_{1}(x)$, for all $x$, and
			$${\begin{bmatrix}
					k_{1}^2(x) & 2r(k_{2}(x)-k_{1}(x))d^3(x)\\
					2r(k_{2}(x)-k_{1}(x))d^3(x) & k_{2}^2(x)+4r(k_{2}(x)-k_{1}(x))d^4(x)
			\end{bmatrix}}$$
			is positive semi-definite for all $x$, or
			\item if $\lambda_{1}(x)-\lambda_{2}(x)$ changes sign, and the above two matrices are positive semi-definites for all $x$,
		\end{enumerate}	
		and the cross-diffusion matrix is of the form $g(\cdot)M$ for some non-negative real valued $g$ and positive semi-definite $M$, then the stable block $LU$ factorization (\ref{blockLU_factorization}) exists and it is unique for any choice of $r$ in (\ref{r_aos})-(\ref{r_amos}).
		Furthermore, for a general positive semi-definite, not necessarily symmetric, cross-diffusion matrix, the stable block $LU$ factorization (\ref{blockLU_factorization}) exists and is unique, for sufficiently small $r$ in (\ref{r_aos})-(\ref{r_amos}), as long as we change the requirements on the matrices in conditions 2-4 to positive definiteness.
	\end{theorem}
	\begin{proof}
		Existence and uniqueness follow by the previous lemmata. In this regard, all that it remains to check is that conditions 2-4 in Lemma 4 imply conditions 2-4 in Lemma 5. Considering matrix (\ref{lemma4_matrix1}), notice that taking $\kappa_1=k_1^2(x)$, $\kappa_2=k_2^2(x)$, $a=2r(k_1(x)-k_2(x))d^1(x)$ and $b=r(k_1(x)-k_2(x))d^2(x)$, if the matrix ${\begin{bmatrix}
				\kappa_1+2a & 2b \\ 2b & \kappa_2
		\end{bmatrix}}$ is positive semi-definite then its determinant is greater or equal to 0. But then $b^2\leq\frac{\kappa_1\kappa_2+2a\kappa_2}{4}$, and
		$$\det{\begin{bmatrix}
				\kappa_1+a & b \\ b & \kappa_2
		\end{bmatrix}}=\kappa_1\kappa_2+\kappa_2a-b^2\geq \kappa_1d_2+\kappa_2a-\frac{\kappa_1\kappa_2+2a\kappa_2}{4}>0,$$
		which, by also having positive diagonals, implies the positive semi-definiteness of matrix (\ref{lemma5_matrix1}). In a similar way the same result is shown for matrix (\ref{lemma5_matrix2}) and for the case of positive definiteness.
		
		It remains to prove the stability of the factorization. We will show that $\|L\|\|U\|\leq c \|A\|$ for some $c>0$, where $A$ is the matrix in (\ref{matrix:blockTridiagonal}) and $LU$ are the matrices in (\ref{blockLU_factorization}). First remember that, for any matrix $B$ with blocks $B_{ij}$, $\max \|B_{ij}\|\leq \|B\| \leq \sum\limits_{i,j} \| B_{ij}\|$. Due to Lemma \ref{lemma:firstmatrixNormsBounds}, $\|\bar{U}_j\|=\|\bar{B}_j^{-1}U_j\|<1$, and therefore $\|U\|<2N_1-1$. Furthermore, as
		\begin{equation*}
			\begin{aligned}
				\|\bar{B}_j\|&=\|B_j-L_{j-1}\bar{B}_j^{-1}\bar{U}_j\|\\
				&=\|B_j(I-\bar{B}_j^{-1}L_{j-1}\bar{B}_j^{-1}\bar{U}_j)\|\\
				&\leq \|B_j\|(\|I\|+\|\bar{B}_j^{-1}L_{j-1}\|\|\bar{B}_j^{-1}\bar{U}_j|)\\
				&\leq 2\|B_j\|,
			\end{aligned}
		\end{equation*}
		it holds that $\|L\|\leq (2N_1-1)\|A\|$. Therefore $\|L\|\|U\|< (2N_1-1)^2\|A\|$.
	\end{proof}
	
	\section{Generalization to three dimensional schemes}
	\label{section:3D}
	
	We now extend the theory developed in the previous sections to three-dimensional cross-diffusion processes. Consider the cross-diffusion system (\ref{cross-diffusion}) where $\Omega=(a_{1},b_{1})\times (a_{2},b_{2})\times (a_3,b_3)\subset\mathbb{R}^{3}$ is now the domain of interest. The remaining definitions for the diffusion process are the same, while the domain $\overline{\Omega}=\Omega\cup\Gamma$ is to be discretized by the points ${\bf x}_{{\bf j}}=(x_{j_{1}},x_{j_{2}},x_{j_3})$, where
	\[
	x_{j_{1}}=a_{1}+h_{1}j_{1},\; x_{j_{2}}=a_{2}+h_{2}j_{2},\; x_{j_3}=a_3+h_3j_3,\quad j_{k}=0,1,\ldots,N_{k},\]
	\[ h_{k}=\frac{b_{k}-a_{k}}{N_{k}}, \; k=1,2,3,
	\]
	for three given integers $N_{1}, N_{2}, N_3\geq 1$, ${\bf j}=(j_{1},j_{2},j_3)$ and ${\bf h}=(h_{1},h_{2},h_3)$. This spatial mesh on $\overline{\Omega}$ is denoted by $\overline{\Omega}_{\bf h}$ and  $\Gamma_{\bf h}=\Gamma \cap \overline{\Omega}_{\bf h}$.  
	Points halfway between two adjacent grid points are denoted by 
	${\bf x}_{{\bf j}\pm (1/2){\bf e}_k}={\bf x}_{\bf j} \pm  \frac{h_k}{2}{\bf e}_{k}$, $k=1,2,3$, where $\{{\bf e}_{1},{\bf e}_{2},{\bf e}_{3}\}$ 
	is the $\mathbb{R}^3$ canonical basis, that is, ${\bf e}_{k}$ is the standard basis unit vector in the $k$th direction. 
	
	The numerical solution of (\ref{cross-diffusion}) at the time $t^{m+1}$ can be obtained considering the following finite difference scheme:
	\begin{equation}
		\label{finite_difference_scheme_3D}
		\begin{cases}
			\frac{U_{\bf j}^{m+1}-U_{\bf j}^{m}}{\Delta t} =
			  \sum\limits_{k=1}^{3}\delta_{k}\left(d^1({\bf W}^m)_{\bf j}^{m+\theta}\delta_{k}U_{{\bf j}}^{m+\theta}+d^2({\bf W}^m)_{\bf j}^{m+\theta}\delta_{k}V_{{\bf j}}^{m+\theta}\right)
			-\lambda_1({\bf W}^m)_{\bf j}^{m+\theta} (U_{\bf j}^{m+\theta}-U_{\bf j}^0), \\
			\frac{V_{\bf j}^{m+1}-V_{\bf j}^{m}}{\Delta t} = 
			 \sum\limits_{k=1}^{3}\delta_{k}\left(d^3({\bf W}^m)_{\bf j}^{m+\theta}\delta_{k}U_{{\bf j}}^{m+\theta}+d^4({\bf W}^m)_{\bf j}^{m+\theta}\delta_{k}V_{{\bf j}}^{m+\theta}\right)-\lambda_2({\bf W}^m)_{\bf j}^{m+\theta} (V_{\bf j}^{m+\theta}-V_{\bf j}^0).
		\end{cases}
	\end{equation}
	
	The realization of AOS and AMOS schemes in the three dimensional setting is straightforward. For the AOS scheme, simply compute and add the third dimension in (\ref{aos_scheme}):
	\begin{subequations}
		\label{aos_scheme_3D}
		\begin{equation}
			\label{aos_3D_fractionalstep}
			\begin{cases}
				\begin{aligned}
					&\frac{U_{\bf j}^{m+1,k}-U_{\bf j}^{m}}{3\Delta t} =
					\delta_{k}\left(d^1({\bf W}^m)_{\bf j}^{m+\theta}\delta_{k}U_{{\bf j}}^{m+\theta}+d^2({\bf W}^m)_{\bf j}^{m+\theta}\delta_{k}V_{{\bf j}}^{m+\theta}\right)-\frac{1}{3}\lambda_1({\bf W}^m)_{\bf j}^{m+\theta} (U_{\bf j}^{m+\theta}-U_{\bf j}^0)\\
					&\frac{V_{\bf j}^{m+1,k}-V_{\bf j}^{m}}{3\Delta t} =
					\delta_{k}\left(d^3({\bf W}^m)_{\bf j}^{m+\theta}\delta_{k}U_{{\bf j}}^{m+\theta}+d^4({\bf W}^m)_{\bf j}^{m+\theta}\delta_{k}V_{{\bf j}}^{m+\theta}\right)-\frac{1}{3}\lambda_2({\bf W}^m)_{\bf j}^{m+\theta} (V_{\bf j}^{m+\theta}-V_{\bf j}^0)
				\end{aligned}
			\end{cases}
		\end{equation}
		for k=1,2,3, and
		\begin{equation}
			\label{aos_3D_averagestep}
			\qquad U_{\bf j}^{m+1}=\frac{1}{3}\sum_{k=1}^{3}U_{\bf j}^{m+1,k},\quad 
			V_{\bf j}^{m+1}=\frac{1}{3}\sum_{k=1}^{3}V_{\bf j}^{m+1,k}.
		\end{equation}
	\end{subequations}
	
	For the AMOS scheme we require some extra sets of systems. The number of full systems to be solved for a $d$-dimensional cross diffusion process is $d\cdot d!$:
	\begin{subequations}
		\label{amos_scheme_3D}
		\begin{equation}
			\label{amos_3D_factorizationstep1}
			\begin{cases}
				\begin{aligned}
					&\frac{U_{\bf j}^{m+1,k}-U_{\bf j}^{m}}{\Delta t} =
					\delta_{k}\left(d^1({\bf W}^m)_{\bf j}^{m+\theta}\delta_{k}U_{{\bf j}}^{m+\theta}+d^2({\bf W}^m)_{\bf j}^{m+\theta}\delta_{k}V_{{\bf j}}^{m+\theta}\right) -\frac{1}{3}\lambda_1({\bf W}^m)_{\bf j}^{m+\theta} (U_{\bf j}^{m+\theta}-U_{\bf j}^0) \\
					&\frac{V_{\bf j}^{m+1,k}-V_{\bf j}^{m}}{\Delta t} =
					\delta_{k}\left(d^3({\bf W}^m)_{\bf j}^{m+\theta}\delta_{k}U_{{\bf j}}^{m+\theta}+d^4({\bf W}^m)_{\bf j}^{m+\theta}\delta_{k}V_{{\bf j}}^{m+\theta}\right)-\frac{1}{3}\lambda_2({\bf W}^m)_{\bf j}^{m+\theta} (V_{\bf j}^{m+\theta}-V_{\bf j}^0)
				\end{aligned}
			\end{cases}
		\end{equation}
		for $k=1,2,3$,
		\begin{equation}
			\label{amos_3D_factorizationstep2}
			\begin{cases}
				\begin{aligned}
					&\frac{U_{\bf j}^{m+1,ki}-U_{\bf j}^{m+1,k}}{\Delta t} =
					\delta_{i}\left(d^1({\bf W}^m)_{\bf j}^{m+\theta}\delta_{i}U_{{\bf j}}^{m+\theta,k}+d^2({\bf W}^m)_{\bf j}^{m+\theta}\delta_{i}V_{{\bf j}}^{m+\theta,k}\right) -\frac{1}{3}\lambda_1({\bf W}^m)_{\bf j}^{m+\theta} (U_{\bf j}^{m+\theta,k}-U_{\bf j}^0)\\
					&\frac{V_{\bf j}^{m+1,ki}-V_{\bf j}^{m+1,k}}{\Delta t} =
					\delta_{i}\left(d^3({\bf W}^m)_{\bf j}^{m+\theta}\delta_{i}U_{{\bf j}}^{m+\theta,k}+d^4({\bf W}^m)_{\bf j}^{m+\theta}\delta_{i}V_{{\bf j}}^{m+\theta,k}\right)-\frac{1}{3}\lambda_2({\bf W}^m)_{\bf j}^{m+\theta} (V_{\bf j}^{m+\theta,k}-V_{\bf j}^0)
				\end{aligned}
			\end{cases}
		\end{equation}
		for $k,i=1,2,3$, $i\neq k$, and
		\begin{equation}
			\label{amos_3D_factorizationstep3}
			\begin{cases}
				\begin{aligned}
					&\frac{U_{\bf j}^{m+1,kij}-U_{\bf j}^{m+1,ki}}{\Delta t} =
					\delta_{j}\left(d^1({\bf W}^m)_{\bf j}^{m+\theta}\delta_{j}U_{{\bf j}}^{m+\theta,ki}+d^2({\bf W}^m)_{\bf j}^{m+\theta}\delta_{j}V_{{\bf j}}^{m+\theta,ki}\right) -\frac{1}{3}\lambda_1({\bf W}^m)_{\bf j}^{m+\theta} (U_{\bf j}^{m+\theta,ki}-U_{\bf j}^0)\\
					&\frac{V_{\bf j}^{m+1,kij}-V_{\bf j}^{m+1,ki}}{\Delta t} =
					\delta_{j}\left(d^3({\bf W}^m)_{\bf j}^{m+\theta}\delta_{j}U_{{\bf j}}^{m+\theta,ki}+d^4({\bf W}^m)_{\bf j}^{m+\theta}\delta_{j}V_{{\bf j}}^{m+\theta,ki}\right)-\frac{1}{3}\lambda_2({\bf W}^m)_{\bf j}^{m+\theta} (V_{\bf j}^{m+\theta,ki}-V_{\bf j}^0)
				\end{aligned}
			\end{cases}
		\end{equation}
		for $k,i,j=1,2,3$, $j\neq i\neq k$, and finally average
		\begin{equation}
			\label{amos_3D_average}
			U_{\bf j}^{m+1}=\frac{1}{6}\sum_{i\neq j \neq k} U_{\bf j}^{m+1,kij},\qquad
			V_{\bf j}^{m+1}=\frac{1}{6}\sum_{i\neq j \neq k} V_{\bf j}^{m+1,kij}.
		\end{equation}
	\end{subequations}
	
	As for the computational considerations of 3D schemes, the full implicit method (\ref{finite_difference_scheme}) requires the solution of a linear system of size $2N_1N_2N_3$, while the three-dimensional AOS-CD (\ref{aos_scheme_3D}) requires the solution of $N_2N_3$ linear systems of size $2N_1$, $N_1N_3$ linear systems of size $2N_2$, and $N_1N_2$ linear systems of size $2N_3$, and the three-dimensional AMOS-CD scheme (\ref{amos_scheme_3D}) requires the solution of $6N_2N_3$ linear systems of size $2N_1$, $6N_1N_3$ linear systems of size $2N_2$, and $6N_1N_2$ linear systems of size $2N_3$. The systems design is the same as in the end of Section \ref{section:splittings}.
	
	We cluster the extension to the three dimensional case of previous theoretical results into a single theorem and omit the proof, as it follows the same steps as Theorems \ref{thrm:L2_stab_conditions}, \ref{thrm:split_l2_stab} and \ref{thrm:stab_final}. Define now for each $x_{\bf j}=(x_{j_1},x_{j_2},x_{j_3})\in \bar{\Omega}_{\bf h}$ the cube $\square_{\bf j}=(x_{j_1},x_{j_1+1})\times(x_{j_2},x_{j_2+1})\times(x_{j_3},x_{j_3+1})$, denote by $|\square_{\bf j}|$ the measure of $\square_{\bf j}$, and consider the discrete inner products 
	\begin{equation*}\label{inner_product1_3D}
		\begin{aligned}
			({U},{ V})_h=&\sum_{\square_{\bf j}\subset \Omega} \frac{|\square_{\bf j}|}{8} \big({ U}_{j_1,j_2,j_3}{ V}_{j_1,j_2,j_3} +{ U}_{j_1+1,j_2,j_3}{{ V}}_{j_1+1,j_2,j_3}\\
			&\quad +{U}_{j_1,j_2+1,j_3}{V}_{j_1,j_2+1,j_3}+{ U}_{j_1+1,j_2+1,j_3}{{ V}}_{j_1+1,j_2+1,j_3}\\
			&\quad +{ U}_{j_1,j_2,j_3+1}{ V}_{j_1,j_2,j_3+1} +{ U}_{j_1+1,j_2,j_3+1}{{ V}}_{j_1+1,j_2,j_3+1}\\
			&\quad +{U}_{j_1,j_2+1,j_3+1}{V}_{j_1,j_2+1,j_3+1}+{ U}_{j_1+1,j_2+1,j_3+1}{{ V}}_{j_1+1,j_2+1,j_3+1}\big)
		\end{aligned}
	\end{equation*}
	\begin{equation*}
		\begin{aligned}
			({U},{V})_{h_1^*}=&\sum_{\square_{\bf j}\subset \Omega} \frac{|\square_{\bf j}|}{4} \big({U}_{j_1+1/2,j_2,j_3}{{V}}_{j_1+1/2,j_2,j_3}+{U}_{j_1+1/2,j_2+1,j_3}{V}_{j_1+1/2,j_2+1,j_3}\\
			&+{U}_{j_1+1/2,j_2,j_3+1}{{V}}_{j_1+1/2,j_2,j_3+1}+{U}_{j_1+1/2,j_2+1,j_3+1}{V}_{j_1+1/2,j_2+1,j_3+1}\big),
		\end{aligned}
	\end{equation*}
	\begin{equation*}
		\begin{aligned}
			({ U},{ V})_{h_2^*}=&\sum_{\square_{\bf j}\subset \Omega} \frac{|\square_{\bf j}|}{4} \big({ U}_{j_1,j_2+1/2,j_3}{{ V}}_{j_1,j_2+1/2,j_3} +{ U}_{j_1+1,j_2+1/2,j_3}{{ V}}_{j_1+1,j_2+1/2,j_3}\\
			&+{ U}_{j_1,j_2+1/2,j_3+1}{{ V}}_{j_1,j_2+1/2,j_3+1} +{ U}_{j_1+1,j_2+1/2,j_3+1}{{ V}}_{j_1+1,j_2+1/2,j_3+1}
			\big),
		\end{aligned}
	\end{equation*}
	and
	\begin{equation*}
		\begin{aligned}
			({ U},{ V})_{h_3^*}=&\sum_{\square_{\bf j}\subset \Omega} \frac{|\square_{\bf j}|}{4} \big({ U}_{j_1,j_2,j_3+1/2}{{ V}}_{j_1,j_2,j_3+1/2} +{ U}_{j_1+1,j_2,j_3+1/2}{{ V}}_{j_1+1,j_2,j_3+1/2}\\
			&+{ U}_{j_1,j_2+1,j_3+1/2}{{ V}}_{j_1,j_2+1,j_3+1/2} +{ U}_{j_1+1,j_2+1,j_3+1/2}{{ V}}_{j_1+1,j_2+1,j_3+1/2}
			\big).
		\end{aligned}
	\end{equation*}
	
	\begin{theorem}
		
		\
		
		\begin{enumerate}
			\item If the cross-diffusion matrix (\ref{cross-diffusion matrix}) is such that, for any $U,V\in\mathbb{R}^{N_1\times N_2 \times N_3}$,
			\begin{equation}\label{cdts:functions_semiimplicit_3D}
				\begin{aligned}
					\sum_{k=1}^3 (d^1\delta_k U, \delta_k U)_{h_k^*}+(d^2\delta_k V, \delta_k U)_{h_k^*}+(d^3\delta_k U, \delta_k V)_{h_k^*}+(d^4\delta_k V, \delta_k V)_{h_k^*}\geq 0
				\end{aligned}
			\end{equation}
			then the scheme (\ref{finite_difference_scheme_3D}) is unconditionally stable for $\theta\in[\frac{1}{2},1]$.
			
			If the functions $d^1,d^2,d^3$ and $d^4$ are such that, for any $U,V\in\mathbb{R}^{N_1\times N_2\times N_3}$,
			\begin{equation}
				\begin{aligned}
					\label{cdts:functions_explicit_3D}
					\sum_{k=1}^3 &(d^1\delta_k U, \delta_k U)_{h_k^*}+(d^2\delta_k V, \delta_k U)_{h_k^*}+(d^3\delta_k U, \delta_k V)_{h_k^*}+(d^4\delta_k V, \delta_k V)_{h_k^*}\\
					&-\Big( \frac{4\Delta t}{h_k^2}\big((1+\eta_1)(\|d^1\delta_{k}U\|_{h_k^*}^2+\|d^2\delta_{k}V\|_{h_k^*}^2 +2(d^1\delta_{k}U,d^2\delta_{k}V)_{h_k^*})\\
					&+(1+\eta_2)(\|d^3\delta_{k}U\|_{h_k^*}^2+\|d^4\delta_{k}V\|_{h_k^*}^2+2(d^3\delta_{k}U,d^4\delta_{k}V)_{h_k^*}) \Big) \geq 0,
				\end{aligned} 
			\end{equation}
			for some $\eta_1,\eta_2 >0$, then the scheme (\ref{finite_difference_scheme_3D}) is stable for $\theta=0$.
			\item If the cross-diffusion matrix (\ref{cross-diffusion matrix}) is such that, for any $U,V\in\mathbb{R}^{N_1\times N_2\times N_3}$, condition (\ref{cdts:aos_imp_stab}) is satisfied
			then the schemes (\ref{aos_scheme_3D}) and (\ref{amos_scheme_3D}) are unconditionally stable for the case $\theta\in[\frac{1}{2},1]$.
			
			If the cross-diffusion matrix (\ref{cross-diffusion matrix}) is such that, for any $U,V\in\mathbb{R}^{N_1\times N_2\times N_3}$, condition (\ref{cdts:aos_exp_stab}) is satisfied for some $\eta_1,\eta_2 >0$, then the schemes (\ref{aos_scheme_3D}) and (\ref{amos_scheme_3D}) are stable for $\theta=0$.
			
			\item If the functions $\lambda_1(\cdot)$, $\lambda_2(\cdot)$ and the cross-diffusion matrix (\ref{cross-diffusion matrix}) are such that one of conditions $1-4$ of Theorem \ref{thrm:stab_final} is satisfied, and the cross-diffusion matrix is of the form $g(\cdot)M$ for some non-negative real valued function $g$ and positive semi-definite matrix $M$, then the stable block $LU$ factorization (\ref{blockLU_factorization}) exists and it is unique for any choice of $r$ in (\ref{r_aos})-(\ref{r_amos}).
			Furthermore, for a general positive semi-definite, not necessarily symmetric, cross-diffusion matrix, the stable block $LU$ factorization (\ref{blockLU_factorization}) exists and is unique, for sufficiently small $r$ in (\ref{r_aos})-(\ref{r_amos}), as long as we change the requirements on the matrices in conditions 2-4 to positive definiteness.
		\end{enumerate}
	\end{theorem}
	
	\section{Speedup - operation count and numerical experiments}
	\label{section:speedups}
	
	We now discuss the speedup gained by using either AOS-CD or AMOS-CD, both with and without the banded algorithm. A pass in the first direction is described in Algorithm \ref{alg:block_tridiagonal}. For the other directions the algorithm is analogous.
	
	\begin{algorithm}
		\label{alg:blockLUPass}
		\caption{Block tridiagonal row iteration in fully implicit split cross-diffusion}
		{\begin{algorithmic}[1]
				\Require Vectors $U_{1,i}^m,\dots,U_{N_1,i}^m$, $V_{1,i}^m,\dots,V_{N_1,i}^m$ and evaluated mid-pixels $d^\ell({\bf W}^m)_{N_1-1/2,i},\dots,d^\ell({\bf W}^m)_{1+1/2,i}$, for $\ell=1,\dots,4$.
				\State Set $w^m$ as in (\ref{w_ordered_blockTridiagonal});
				\State Solve $B_{1,i}x_1=w^m_1$. Here, $x_1$ ($w_1^m$) stands for the first pair of elements in $x$ ($w^m$), $x_2$ ($w_2^m$) for the second pair in $x$ ($w^m$), etc;
				\State Set $w^m_1=x_1$;
				\For{$k=2,\dots,N_1$}
				\State Solve $B_{k-1,i}\bar U_{k-1,i}=U_{k-1,i}$;
				\State Set $B_{k,i}=B_{k,i}-L_{k-1,i}\bar U_{k-1,i}$;
				\State Set $w_k^m=w_k^m-L_{k-1}w_{k-1}$;
				\State Solve $B_kx_k=w_k^m$;
				\State Set $w_k^m=x_k$;
				\EndFor
				\State Set $w_{N_1}^{m+1}=x_{N_1}^m$;
				\For{$k=N_1,\dots,2$}
				\State Set $w_{k-1}^{m+1}=w_{k-1}^m-\bar U_{k-1,i}w_k^m$;
				\EndFor
				\Ensure Filtered rows $U_{1,i}^{m+1},\dots,U_{N_1,i}^{m+1}$ and $V_{1,i}^{m+1},\dots,V_{N_1,i}^{m+1}$. Optionally, save also factorization terms $B_{2,i}\bar{U}_{1,i},\dots,B_{N_1,i}\bar{U}_{N_1-1,i}$ (for details see the discussion in Section \ref{section:speedups}).
		\end{algorithmic}}\label{alg:block_tridiagonal}
	\end{algorithm}
	
	The standard $LU$ solver for a system with $n$ equations and unknowns, disregarding lookups for pivoting, requires a total of $\frac{2}{3}n^3+\frac{3}{2}n^2-\frac{1}{6}n$ flops ($\frac{1}{6}n(n-1)(4n+1)$ for the factorization, $n(n-1)$ flops for the forward substitution and $n^2$ flops for the backward substitution \cite{Golub13}). Therefore, for a squared domain discretized with a grid of size $N$ in each direction, and disregarding the explicit computations, the $\theta$-method requires $\frac{2}{3}N^6+\frac{3}{2}N^4-\frac{1}{6}N^2$ flops, while a pass in a single direction of both AOS-CD and AMOS-CD requires $\frac{16}{3}N^3+6N^2-\frac{1}{3}N$ flops, resulting in a total of $\frac{32}{3}N^4+12N^3+\frac{4}{3}N^2$ flops for an AOS iteration and $\frac{64}{3}N^4+24N^3+\frac{2}{3}N^2$ for an AMOS iteration.
	
	Taking now into consideration the banded solver for a squared domain of size $N$, a pass in each direction of both AOS-CD and AMO S-CD requires $(18+48)(N-1)$ flops for the factorization (corresponding, respectively, to rows 5 and 6 in Algorithm (\ref{alg:block_tridiagonal})), $9+(12+9)(N-1)$ flops for the forward substitution (rows 2, 7 and 8) and $12(N-1)$ flops for the back substitution (rows 11 and 13). This amounts to a total of $200N^2+180N$ flops for a banded AOS-CD iteration and $398N^2+360$ flops for a banded AMOS-CD iteration. However, notice that the $LU$ factorization is independent of the right hand side. Although in AOS-CD each direction pass is carried out only once, in AMOS-CD each direction is passed twice (in a two-dimensional domain) or more. Therefore, as the factorization takes roughly two thirds of the number of flops of a pass, we believe that non-negligible gains could be obtained if each factorization was saved for further use (that is, if the system has no memory limitations for the current domain, in the sense that it its performance is not hindered by the allocation of an extra $4(N-1)N$ array for the two-dimensional case, or an extra $6(N-1)N^2$ array for the three-dimensional case). However, this situation requires some extra care to optimize the system's reading and inter-function communications, and we not take it into account in our numerical experiments.
	
	Similar calculations can be performed for non-square or three-dimensional domains. The flop loads, characterized by the factors with the highest magnitude, are described in Table \ref{tab:speedups}. Numerical tests performed with double precision in a Ryzen 7 3700X @ 3600 Mhz with 16 GB of RAM can be seen in Figure \ref{fig:speedups}. The linear system solutions were obtained through the Matlab backslash operator "$\backslash$", while the iterative banded procedure was implemented as a mex call using the Matlab Coder Toolbox for a fair comparison with the highly optimized embedded system solvers. The experiments show an impressive improvement on running time, averaging a speed-up for two-dimensional processes of magnitude $4$, $2$, $100$, and $50$ with, respectively, schemes AOS-CD, AMOS-CD, AOS-CD Banded, and AMOS-CD Banded, and averaging a speed-up for three-dimensional processes of magnitude $70$, $11$, $3300$, and $600$ with, respectively, the same schemes. Notice also that the running times of Banded iterations remain close to explicit implementations ($\theta=0$ in (\ref{finite_difference_scheme}) and (\ref{finite_difference_scheme_3D})).

	\begin{table}[!htb]
		\centering
		\begin{tabular}{lllll}
			& \textit{\small 2D Squared} & \textit{\small 3D Cubic} & \textit{\small 2D General}               & \textit{\small 3D General}  \\ \hline \vspace{0.025in}
			\textit{ Fully Implicit}        & $ \frac{2}{3}N^6$    & $ \frac{2}{3}N^9$  & $ \frac{2}{3}N_1^3N_2^3$           & $ \frac{2}{3}N_1^3N_2^3N_3^3$                        \\ \hline
			\textit{ AOS-CD}         & $ \frac{32}{3}N^4$   & $ 16 N^5$          & $ \frac{16N_1N_2}{3}\sum_{i=1}^2 N_k^2$ & $ \frac{16N_1N_2N_3}{3} \sum_{i=1}^3 N_k^2$ \\ \hline
			\textit{ AMOS-CD}        & $ \frac{64}{3}N^4$   & $ 96 N^5$          & $ \frac{32N_1N_2}{3}\sum_{i=1}^2 N_k^2$  & $ \frac{96N_1N_2N_3}{3}\sum_{i=1}^3 N_k^2$ \\ \hline
			\textit{ AOS-CD Banded}  & $ 200N^2$            & $ 300 N^3$         & $ 200N_1N_2$                       & $ 300 N_1N_2N_3$                                     \\ \hline
			\textit{ AMOS-CD Banded} & $ 398N^2$            & $ 1195 N^3$        & $ 398N_1N_2$                       & $ 1195 N_1N_2N_3$                                    \\ 
		\end{tabular}
	\vspace{5pt}
		\caption{Flop counts of the factors with highest magnitude for each type of implicit implementation and different domains.}
		\label{tab:speedups}
	\end{table}

	\begin{figure*}[!htb]
		\center
		\minipage{0.5\textwidth}
		\includegraphics[width=\linewidth]{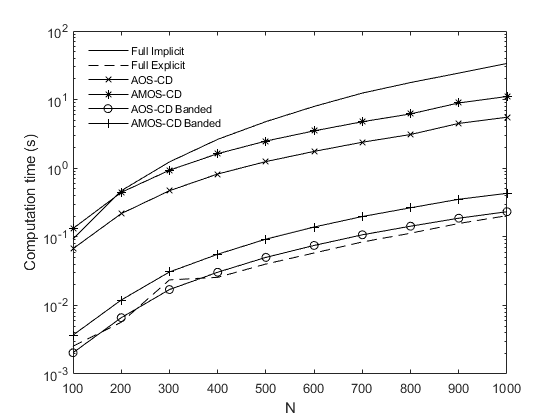}
		\endminipage
		\minipage{0.5\textwidth}
		\includegraphics[width=\linewidth]{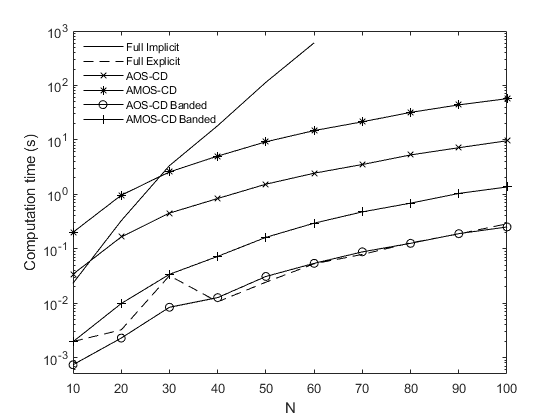}
		\endminipage
		\caption{Implicit iteration computation times  for 2D and 3D splitting methods on square (left) and cubic (right) domains. Missing markers indicate impracticable computations for the test machine.}
		\label{fig:speedups}
	\end{figure*}
	
	\section{Concluding remarks}
	We proposed two stable splitting techniques for cross-diffusion processes in two and three dimensions. Using a special factorization of the system matrix, we propose fast and stable algorithms for implicit cross-diffusion schemes, allowing them to become competitive against their explicit counterparts regarding computation time. This fact empowers the use of cross-diffusion processes in multi-dimensional applications that require on-the-fly results, such as image denoising or image segmentation.
	
	Future work, inspired by these results, includes the study of convergence and consistency properties of the schemes considered in this paper, their application in GPU computing architectures and their use in machine learning frameworks.
	
	\section*{Acknowledgments}
	{\small
		The author was partially supported by the Centre for Mathematics of the University of Coimbra - UIDB/00324/2020, funded by the Portuguese Government through FCT/MCTES, and by the FCT grant PD/BD/142956/2018.}
	
	\newcommand{\etalchar}[1]{$^{#1}$}

\end{document}